\documentclass[11pt,a4paper,twoside]{article}      % Comments after  % are ignored
\usepackage{amsmath,amssymb,amsfonts} % Typical maths resource packages
\usepackage{graphicx}                 % Packages to allow inclusion of graphics
\usepackage{color}                    % For creating coloured text and background
\usepackage{hyperref}                 % For creating hyperlinks in cross references
\usepackage[english]{babel}
\usepackage{float,subfig}
\usepackage{textcomp}
\newenvironment{proof}{\paragraph{\it Proof.}}{\hfill$\square$}
\newcommand{\npar}{\par \vspace{2.3ex plus 0.3ex minus 0.3ex}}
\newtheorem{theorem}{Theorem}
\newtheorem{proposition}{Proposition}
\newtheorem{corollary}{Corollary}

\newtheorem{claim}{Claim}
\newtheorem{remark}{Remark}
\newtheorem{definition}{Definition}

\title{\bf Classification and study of a new class  of $ \xi^{(as)} $-QSO }

\author{\bf Hamza Abd El-Qader$^{1}$, \bf Ahmad Termimi Ab Ghani	$^{1}$, \bf I. Qaralleh$^{2}$\\
\small $^{1}\thanks{Corresponding author: hamza88q@gmail.com}$ \it School of Informatics and Applied Mathematics 
\\
\small \it Universiti Malaysia Terengganu\\
\small \it 21030 Kuala Nerus \\
\small\it Terengganu\\
\small $^{2}$ \it Department of Mathematics, Faculty of Science\\
\small \it Tafila Technical University\\
\small \it Tafila 66110, Jordan\\
}

\begin{document}
\maketitle

{\bf Abstract --}
{ Many systems are presented using theory of nonlinear operators. A quadratic stochastic operator (QSO) is perceived as a nonlinear  operator. It has a wide range of applications in various disciplines, such as mathematics, biology, and other sciences. The central problem that surrounds this nonlinear operator lies in the requirement that  behavior should be studied. Nonlinear operators, even QSO (i.e., the simplest nonlinear operator), have not been thoroughly investigated. This study aims  to present a new class of $\xi^{(as)}$-QSO defined on 2D simplex and to classify it into 18 non-conjugate (isomorphic) classes based on their conjugacy and the remuneration of  coordinates. In addition,  the limiting points of the behavior of trajectories for four  classes  defined on 2D simplex are examined.}
\npar 
%{\bf Keywords:}
%Fixed point, limiting point, quadratic stochastic operator.
\npar 
%{\bf Mathematics Subject Classification:} 37E99; 37N25; 39B82; 47H60; 92D25.
\section{Introduction}

The concept of  a quadratic stochastic operator (QSO) was developed by brainchild of S. Bernstein in 1924  \cite{1}. Since then, QSOs have been  studied  intensively as they emerge in various models in physics \cite{24,32}, biology \cite{1,15,14}, economics and different branches of mathematics, such as graph theory and probability theory \cite{15,17,20,33}.\\ 

In the biological context,  QSOs can be applied  in the area of population genetics. QSO can describe a distribution of the next generation when the initial distribution of the generation is provided. We shall briefly highlight how these operators are used to interpret  in population genetics. Consider a biological population, i.e., a community of organisms that is closed with regard to procreation. Assume that every individual in this population belongs to one of the following varying species (traits): $\lbrace1,\cdots, m\rbrace $. Let $x^{(0)}=(x_1^{(0)},\cdots,x_m^{(0)})$ be a probability distribution of species at an initial state and let  the heredity coefficient $ p_{ij,k} $ be the conditional probability  $ p(k \backslash i,j ) $ that $ i^{th} $ and  $ j^{th} $ species have interbred successfully to produce an individual $ k^{th}$. The first generation $x^{(1)}=\left(x_{1}^{(1)},\cdots,x_{m}^{(1)}\right)$ can be calculated using the total probability $ x_k^{(1)}=\sum_{i,j=1}^m   p(k \backslash i,j ) P(i,j), \quad k=\overline{1,m}$. “Given that  no difference exists between  $ i^{th} $ and  $ j^{th} $   in any generation, the “parents” $ i,j $ are independent, i.e., $ P(i,j)=P_{i}P_{j} $. This condition suggests that  \begin{equation*} x_k^{(1)}=\sum_{i,j=1}^m P_{ij,k}  x_i^{(0)}  x_j^{(0)}, \quad k=\overline{1,m}.
\end{equation*}\\
Consequently, the relation  $ x^{(0)}\rightarrow x^{(1)} $ represents  a mapping $V$, which is known as the evolution operator.  Starting  from the selected initial state $ x^{(0)} $, the population  develops to the first generation $ x^{(1)}=V(x^{(0)}) $, and then to the second generation $x^{(2)}=V(x^{(1)})=V(V(x^{(0)}))=V^{(2)}\left(x^{(0)}\right)$, and so on. Hence, the discrete dynamical system presents  population system evolution states as follows: 
$$x^{(0)}, \quad x^{(1)}=V\left(x^{(0)}\right), \quad x^{(2)}=V^{(2)}\left(x^{(0)}\right), \quad  \cdots .$$ One of the main issues that underlies this theory  is  finding  the limit points of $ V $ for any  arbitrary initial point $ x^{(0)} $. Studying  the limit points of QSOs is a complicated task even in 2D simplex. This problem has not yet been  solved. Numerous researchers have presented a specific class of QSO and have examined  their behavior, e.g., F-QSO \cite{28}, Volterra-QSO \cite{7,8,9}, permutated Volterra-QSO \cite{12,13},  $\ell$-Volterra-QSO \cite{26,27}, Quasi-Volterra-QSO \cite{5}, non-Volterra-QSO \cite{6,31}, strictly non-Volterra-QSO \cite{29},  non-Volterra operators, and  others produced via measurements  \cite{3,4}. An attempt was made to study the behavior of nonlinear operators, which is regarded  as the main problem in nonlinear operators.  However this problem has not been studied comprehensively  because it depends on a given cubic matrix  $(P_{ijk})_{i,j,k=1}^m$. Nevertheless,   these classes together cannot cover a set of
all QSOs. \\
\\
Recently, the author of  \cite{36} introduced $\xi^{(as)}$-QSO, which is a new class of QSOs that depend on a partition of the coupled index sets (which have couple traits) ${\bf P}_{m}=\{(i,j): i<j\}\subset I\times I$ and  ${\bf\Delta}_{m}=\{(i,i): i\in N\}\subset I\times I$. In case of 2D  simplex $(m=3)$, ${\bf P}_{3}$ and $\bf \Delta_{m} $  have five possible partitions.
\\
\\
In \cite{36,41}, the $\xi^{(s)}$-QSO related to  $| \xi_1 |=2 $ of  $\bf P_{3} $  with point a partition of $\bf \Delta_{3} $ was investigated. In \cite{22}, the $\xi^{(a)}$-QSO related to  $| \xi_1 |=2 $ of  $\bf P_{3} $  with a trivial  partition of $\bf \Delta_{3} $ was studied.  The  $\xi^{(as)}$-QSO   related to  $| \xi_1 |=3 $ of  $\bf P_{m} $  with a point partition of  $\bf \Delta_{3} $ was examined in \cite{21}. Furthermore, the  $\xi^{(s)}$-QSO and  $\xi^{(a)}$-QSO  related to $| \xi_1 |=1 $ of $\bf P_{3} $  with point and trivial partitions  of  $\bf \Delta_{3} $, respectively, were  discussed  in\cite{391}. Therefore,  some partitions of $\bf \Delta_{3} $   which have not yet been studied. The  current work describes and classifies the operators generated  by $\xi^{(as)}$-QSO  with a cardinality $| \xi_{i} |=2 $ of $\bf P_{3} $  and $\bf \Delta_{3} $  generated  also by  $| \xi_{i} |=2 $. The rest of this paper is organized as follows. Section 2 establishes a number of preliminary definitions.  Section 3  presents the description and classification of  $\xi^{(as)}$-QSO. Section 4 elucidates  the study  examines the behavior of $V_{3}$ and $V_{15}$ obtained from classes $ G_{3} $ and $ G_{9} $, respectively. Section 5 examines  the behavior of  $V_{26}$ and $V_{25}$ obtained from classes $ G_{13} $ and $ G_{14} $, respectively. 

\section{Preliminaries}
Several basic concepts are recalled in this section.
\begin{definition}  
QSO is a mapping of the simplex
\begin{eqnarray}\label{p.1}
S^{m-1}=\left\{x=(x_{1},\cdots, x_{m})\in\mathbb{R}^m:
\ \ \sum_{i=1}^m x_{i}=1, \ \  x_{i}\geq0, \ \ i=\overline{1,m} \right\}
\end{eqnarray}
into itself with the form
\begin{eqnarray}\label{p.2}
x'_{k}=\sum_{i,j=1}^mP_{ij,k}x_{i}x_{j}, \ \ \ k=\overline{1,m},
\end{eqnarray}
where $V(x)=x'=(x'_1,\cdots,x'_m)$, and $P_{ij,k}$ is a coefficient of heredity that satisfies the following conditions:
\begin{eqnarray}\label{p.3}
P_{ij,k}\geq0,\quad P_{ij,k}=P_{ji,k},\quad \sum_{k=1}^m P_{ij,k}=1.
\end{eqnarray}
\end{definition}
From the preceding definition, we can conclude that each QSO  $V:S^{m-1}\to S^{m-1}$ can be uniquely defined by a cubic matrix ${\mathcal{P}}=\big(P_{ijk}\big)_{i,j,k=1}^m$ with conditions \eqref{p.3}.
\\
\\
For $V:S^{m-1}\to S^{m-1}$, we denote the set of fixed points as $Fix(V)$. Moreover, for $x^{(0)}\in S^{m-1}$, we denote the set of limiting points as ${\omega_{V}}(x^{(0)})$. \\

Recall that  Volterra-QSO is defined by \eqref{p.2}, \eqref{p.3}, and the additional assumption
\begin{equation}\label{Vol}
P_{ij,k}=0  \quad \text{if}\quad  k\not\in \{i,j\}. 
\end{equation}

The biological treatment of Condition \eqref{Vol} is clear: \textit{the offspring repeats the genotype (trait) of one of its parents.}. Volterra-QSO exhibits the following form:
\begin{equation}\label{Vol2}
x'_k=x_k\left(1+\sum^m_{i=1}a_{ki}x_i\right),\ \ k\in
I,
\end{equation} 
where
\begin{equation}\label{p.8}
a_{ki}=2P_{ik,k}-1 \ \ \mbox{for}\, i\neq k\,\,\mbox{and}
\ \,a_{ii}=0, \ i \in I.
\end{equation}
Moreover,
  
$$a_{ki}=-a_{ik}\ \ \mbox{and} \ \ |a_{ki}| \leq 1.$$

This type of operator was intensively studied in  \cite{7,8,9}.
\\
\\
The concept of $\ell$-Volterra-QSO was introduced in \cite{26}. This concept is recalled as follows. 
\\
\\
Let $\ell \in I$ be fixed.  Suppose that the heredity coefficient $\left\lbrace P_{ij,k}\right\rbrace $ satisfies
\begin{equation}\label{p.11}
 P_{ij,k}= 0 \ \ \mbox{if} \ \ k \not\in \{i,j\}\ \ \mbox{for }\ \mbox{any}\ \ k\in \{1,\dots,\ell\},\ \ i,j \in I,
\end{equation}
\begin{equation}\label{p.11a}
 P_{i_0j_0,k}> 0 \ \ \text{for some} \  (i_0,j_0),\  i_0\neq
k,\ j_0\neq k, \ \  k \in
\{\ell+1,\dots,m\}.
\end{equation}

Therefore, the QSO defined by \eqref{p.2}, \eqref{p.3},
\eqref{p.11}, and \eqref{p.11a} is called {\it $\ell$-Volterra-QSO}.

\begin{remark} Here, we emphasize the following points:
\begin{enumerate}
\item[(i)] An $\ell$-Volterra-QSO is a Volterra-QSO if and only if
$\ell= m$.

\item[(ii)] No periodic
trajectory exists for Volterra-QSO \cite{7}. However,
such trajectories exist for $\ell$-Volterra-QSO \cite{26}.
\end{enumerate}
\end{remark}

In accordance with \cite{36}, each element $x\in S^{m-1}$ is a probability distribution of  set $I=\{1,...,m\}$. Let $x=(x_{1},\cdots,x_{m})$ and $y=(y_{1},\cdots,y_{m})$ be vectors obtained from $S^{m-1}$. We say that {\it $x$ is equivalent to $y$} if $x_{k}=0$ $\Leftrightarrow$ $y_{k}=0$. We denote this relation as $x\sim y$.
\\
\\
Let $supp(x)=\{i:x_{i}\neq0\}$ be a support of $x\in S^{m-1}$. We say that {\it $x$ is singular to  $y$} and denote this relation as $x\perp y$ if $supp(x)\cap supp(y)=\emptyset.$ Notably  if $x,y\in S^{m-1}$, then $x\perp y$ if and only if $(x,y)=0$, where $(\cdot,\cdot)$ denotes  a standard inner product in $\mathbb{R}^m$.
\\
\\
We denote sets of coupled indexes as
$${\bf P}_m=\{(i,j):\ i<j\}\subset I\times I, \quad \Delta_m=\{(i,i): i\in I\}\subset I\times I.$$
For a given pair $(i,j)\in \mathbf{P}_m\cup \Delta_m$,  a  vector $\mathbb{P}_{ij}=\left(P_{ij,1},\cdots, P_{ij,m}\right)$ is set. Evidently,  $\mathbb{P}_{ij}\in S^{m-1}$.
\\
\\
Let $\xi_1=\{A_{i}\}_{i=1}^N$ and $\xi_2=\{B_{i}\}_{i=1}^M$ be 
fixed partitions of ${\bf P}_m$ and $\Delta_m$, respectively,  i.e.,
$A_{i}\bigcap A_{j}=\emptyset$, $B_{i}\bigcap B_{j}=\emptyset$, 
$\bigcup\limits_{i=1}^N A_{i}=\mathbf{P}_m$, $\bigcup\limits_{i=1}^M
B_{i}=\Delta_m$, where $N,M\leq m$.

\begin{definition}\label{Xi-QSO}\cite{36}
QSO $V:S^{m-1}\to S^{m-1}$ is given by
\eqref{p.2},\eqref{p.3}, is considered a $\xi^{(as)}$-QSO w.r.t. 
partitions $\xi_1$ and $\xi_2$  if the following conditions are
satisfied:
 \begin{enumerate}
\item[(i)] For each $k\in\{1,\dots,N\}$ and any
$(i,j)$, $(u,v)\in A_{k}$,   $\mathbb{P}_{ij}\sim \mathbb{P}_{uv}$is considered.

 \item[(ii)] For any $k\neq \ell$, $k,\ell\in\{1,\dots,N\}$ and
 any $(i,j)\in A_{k}$ and $(u,v)\in A_{\ell}$,   $\mathbb{P}_{ij}\perp\mathbb{P}_{uv}$ is considered.

\item[(iii)] For each $d\in\{1,\dots,M\}$ and any
$(i,i)$, $(j,j)\in B_{d}$,  $\mathbb{P}_{ii}\sim
\mathbb{P}_{jj}$ is considered.

 \item[(iv)] For any $s\neq h$, $s,h\in\{1,\dots,M\}$ and
 any $(u,u)\in B_{s}$ and $(v,v)\in B_{h}$,
  $\mathbb{P}_{uu}\perp\mathbb{P}_{vv}$ is considered.
 \end{enumerate}
\end{definition}
\section{Classification of  $ \xi^{(as)} $- QSO operators}
This section presents the description and classification of $ \xi^{(as)} $-QSO in 2D simplex when $ m=3 $ and the cardinality of the potential partitions of $ \mathbf{P}_m $ and $\bf \Delta_{m}$ are equal to 2. Therefore, the potential partitions of $ \mathbf{P}_3 $ are listed as follows: \\

\begin{eqnarray*}
\xi_1:&=&\{ \lbrace(1,2)\rbrace,\lbrace (1,3)\rbrace, \lbrace( 2,3 )\rbrace \}, \vert \xi_1 \vert =3, \\
\xi_2:&=&\{ \lbrace(2,3) \rbrace,\lbrace (1,2),(1,3)\rbrace \}, \vert \xi_2 \vert =2, \\
\xi_3:&=&\{ \lbrace(1,3) \rbrace,\lbrace (1,2),(2,3)\rbrace \}, \vert \xi_3 \vert =2, \\
\xi_4:&=&\{ \lbrace(1,2) \rbrace,\lbrace (1,3),(2,3)\rbrace \}, \vert \xi_4 \vert =2, \\
\xi_5:&=&\{ (1,2) , (1,3),(2,3)\} ,    \vert \xi_5 \vert=1. \\
\end{eqnarray*}

The potential  partitions of $\bf \Delta_{3}$ are listed as follows: \\
\begin{eqnarray*}
\xi_1:&=&\{ \lbrace(1,1)\rbrace,\lbrace (2,2)\rbrace, \lbrace( 3,3 )\rbrace \}, \vert \xi_1 \vert =3, \\
\xi_2:&=&  \{ (1,1) , (2,2),(3,3)\} , \vert \xi_2 \vert =1,\\
\xi_3:&=&\{ \lbrace(1,1) \rbrace,\lbrace (2,2),(3,3)\rbrace \}, \vert \xi_3 \vert =2, \\
\xi_4:&=&\{ \lbrace(3,3) \rbrace,\lbrace (1,1),(2,2)\rbrace \}, \vert \xi_4 \vert =2, \\
\xi_5:&=& \{ \lbrace(2,2) \rbrace,\lbrace (1,1),(3,3)\rbrace \}, \vert \xi_5 \vert =2. \\
\end{eqnarray*}

\begin{proposition}\label{H1}
For a class of $ \xi^{(as)} $-QSO generated from the possible partitions of $ \mathbf{P}_3 $ and $\bf \Delta_{3}$ with cardinals equal to 2, we determine the following:
\begin{enumerate}
\item[(a)] A class of all $ \xi^{(as)} $-QSO that correspond to  partition $\xi_3  $ of $ \mathbf{P}_3 $ and  partition $\xi_5  $ of $ \bf \Delta_{3}$ is conjugate to a class of all $ \xi^{(as)} $-QSO that correspond to  partition  $\xi_2  $ of $ \mathbf{P}_3 $ and  partition $\xi_3  $ of $ \bf \Delta_{3}$.
\item[(b)] A class of all $ \xi^{(as)} $-QSO that correspond to the partition $\xi_4  $ of $ \mathbf{P}_3 $ and  partition $\xi_4  $ of $ \bf \Delta_{3}$ is conjugate to a classes of all $ \xi^{(as)} $-QSO that correspond to partition  $\xi_2  $ of $ \mathbf{P}_3 $ and  partition $\xi_3  $ of $ \bf \Delta_{3}$.
\end{enumerate}

\end{proposition}

\begin{proof}
$(a)$ In accordance with the general form of QSO given by \eqref{p.2},\eqref{p.3}, the coefficients $ \left( P_{ij,k}\right) _{i,j,k=1}^{3} $  of  operator $ V $ in $ \xi^{(as)} $-QSO that correspond to  partition $ \xi_5= \{ \lbrace(2,2) \rbrace,\lbrace (1,1),(3,3)\rbrace \} $ of $ \bf \Delta_{3} $ and partition $ \xi_3=\{ \lbrace(1,3) \rbrace,\lbrace (1,2),(2,3)\rbrace \} $ of $ \mathbf{P}_3 $  satisfy the following conditions:\\

i. $ \mathbb{P}_{11} \sim \mathbb{P}_{33} $ and $ \mathbb{P}_{22} \bot  \mathbb{P}_{mm} $, $ m=1,3 $; $  \quad$
ii. $ \mathbb{P}_{12} \sim \mathbb{P}_{23} $ and $ \mathbb{P}_{13} \bot  ( \mathbb{P}_{12},\mathbb{P}_{23} )$;\\

where $\mathbb{P}_{ij} =( p_{ij,1},p_{ij,2},p_{ij,3}) $.\\

 To perform $ V_{\pi}=\pi V \pi^{-1} $ transformation on  operator $V $, where permutation $\pi=\begin{pmatrix}
1 & 2 & 3\\
3 & 1 & 2
\end{pmatrix}$.\\
\begin{center}
$ V_{\pi}:x_{k}^{\prime}= \sum_{i,j=1}^3 \mathbb{P}^{\pi}_{ij,k}x_{i}x_{j}, \ \ \ k=\overline{1,3}$,
\end{center}
such that $ \mathbb{P}^{\pi}_{ij,k} =P_{\pi (i) \pi( j), \pi (k)}$, for any $ i,j,k=\overline{1,3} $. Equivalently, $\mathbb{P}^{\pi}_{ij} =\pi \mathbb{P}_{\pi (i) \pi( j)} $ (in vector form) for any $ i,j=1,2,3 $.\\
Subsequently,  operator $ V_{\pi} $ that corresponds to  partitions $ \xi_3 $ of $ \bf \Delta_{3} $ and $ \xi_2 $ of $ \mathbf{P}_3 $ is presented by applying the permutation $ \pi $ for the coefficient of $ V $ that corresponds to  partition $ \xi_5 $ of $ \bf \Delta_{3} $ and $ \xi_3$ of $ \mathbf{P}_3 $. The following relations are derived:
\begin{enumerate}
\item[i.] $\mathbb{P}_{11} \sim \mathbb{P}_{33}  $ and $ \mathbb{P}_{22} \bot ( \mathbb{P}_{11},\mathbb{P}_{33} )$. Given that $ \mathbb{P}^{\pi}_{11}=\mathbb{P}_{33} $, $ \mathbb{P}^{\pi}_{22}=\mathbb{P}_{11} $, and $ \mathbb{P}^{\pi}_{33}=\mathbb{P}_{22} $, we obtain $ \mathbb{P}_{33} \sim \mathbb{P}_{22} $ and $  \mathbb{P}_{11} \bot ( \mathbb{P}_{22},\mathbb{P}_{33} ) $. \\ 
\item[ii.]
$\mathbb{P}_{12} \sim \mathbb{P}_{23}  $ and $ \mathbb{P}_{13} \bot ( \mathbb{P}_{12},\mathbb{P}_{23} )$. Given that $ \mathbb{P}^{\pi}_{12}=\mathbb{P}_{13} $, $ \mathbb{P}^{\pi}_{13}=\mathbb{P}_{23} $, and $ \mathbb{P}^{\pi}_{23}=\mathbb{P}_{12} $, we obtain $ \mathbb{P}_{12} \sim \mathbb{P}_{13} $ and $  \mathbb{P}_{23} \bot ( \mathbb{P}_{12},\mathbb{P}_{13} ) $. \\ 
\end{enumerate}

Similarly, we can prove b by choosing  permutation $\pi=\begin{pmatrix}
1 & 2 & 3\\
2 & 3 &1 
\end{pmatrix}$. This process completes the proof.
\end{proof}
\\

The preceding discussion shows that any $ \xi^{(as)} $-QSO obtained from the class   that corresponds
to partitions $ \xi_5 $ of $ \bf \Delta_{3} $ and $ \xi_3 $ of $ \mathbf{P}_3 $ or $ \xi_4 $ of $ \bf \Delta_{3} $ and $ \xi_4 $ of $ \mathbf{P}_3 $ is conjugate to certain $ \xi^{(as)} $-QSO obtained from the class that corresponds to partitions $ \xi_3 $ of $ \bf \Delta_{3} $ and $ \xi_2 $ of $ \mathbf{P}_3 $.\\

To investigate the operators of  class $ \xi^{(as)} $-QSO that correspond to  partitions  $\xi_2  $ of $ \mathbf{P}_3 $ and $\xi_3  $ of $\bf \Delta_{3}$, coefficient $ \left( P_{ij,k}\right) _{i,j,k=1}^{3} $ in special forms is selected as shown in Tables  (a) and   (b).\\
\begin{center}

(a)\\
\begin{tabular}{c|c||c||c}  \hline 
Case & $P_{11}$ & $P_{22}$ & $P_{33}$  \\ \hline \hline
$I_{1}$ & $\left(\alpha,\beta,0 \right)  $ & $\left(0,0,1 \right)$ & $ \left( 0,0,1 \right) $  \\ \hline
$I_{2}$ & $\left(\alpha,0,\beta \right)  $ & $\left(0,1,0 \right)$ & $ \left( 0,1,0 \right) $  \\ \hline
$I_{3}$ & $\left(\beta,\alpha,0 \right)  $ & $\left(0,0,1 \right)$ & $ \left( 0,0,1 \right) $  \\ \hline
$I_{4}$ & $\left(\beta,0,\alpha\right)  $ & $\left(0,1,0 \right)$ & $ \left( 0,1,0 \right) $  \\ \hline
$I_{5}$ & $\left(0,\alpha,\beta \right)  $ & $\left(1,0,0 \right)$ & $ \left( 1,0,0  \right) $  \\ \hline
$I_{6}$ & $\left(0,\beta,\alpha \right)  $ & $\left(1,0,0  \right)$ & $ \left( 1,0,0  \right) $  \\ \hline
\end{tabular}
\end{center}

where $ \alpha,\beta \in \left[ 0,1\right]  $. Moreover,  $ \alpha+\beta=1$.
\begin{center}
(b)\\
\begin{tabular}{c|c|c|c}  \hline 
Case & $P_{12}$ & $P_{13}$ & $P_{23}$  \\ \hline \hline
$II_{1}$ & $\left(1,0,0 \right)  $ & $\left(1,0,0 \right)$ & $ \left( 0,0,1 \right) $  \\ \hline
$II_{2}$ & $\left(1,0,0 \right)  $ & $\left(1,0,0 \right)$ & $ \left( 0,1,0 \right) $  \\ \hline
$II_{3}$ & $\left(0,1,0 \right)  $ & $\left(0,1,0 \right)$ & $ \left( 1,0,0 \right) $  \\ \hline
$II_{4}$ & $\left(0,1,0 \right)  $ & $\left(0,1,0 \right)$ & $ \left( 0,0,1 \right) $  \\ \hline
$II_{5}$ & $\left(0,0,1 \right)  $ & $\left(0,0,1 \right)$ & $ \left( 1,0,0 \right) $  \\ \hline
$II_{6}$ & $\left(0,0,1 \right)  $ & $\left(0,0,1 \right)$ & $ \left( 0,1,0 \right) $  \\ \hline
\end{tabular}
\end{center}
%Let us introduce the following notations.
%\begin{center}
%$\varphi (x,y,z)=x^2+y^2+z^2, \quad \psi(x,y,z)= 2xy+2xz+2yz  $\\
%$ \varphi (x,y,z)+\psi(x,y,z)=1 $.\\
%$ \varphi (x,y,z)=\varphi (y,z,x)=\varphi (x,z,y)=\varphi (z,y,x)=\varphi (y,x,z)=\varphi (z,x,y) $\\
 %  $\psi(y,z,x)=\psi (x,z,y)=\psi(z,y,x)=\psi (y,x,z)=\psi(z,x,y) $
%\end{center} 
The choices for  Cases $(I_{j},II_{i})$, where $  i,j=1, \cdots ,6 $, provide   36 operators. These operators are defined as follows:

$$
V_{1}:=\left\{
\begin{array}{l}
x'=\alpha (x^{(0)})^2+2x^{(0)}y^{(0)}+2x^{(0)}z^{(0)} \\
y'=\beta (x^{(0)})^2\\
z'=(y^{(0)})^2 +(z^{(0)})^2+2y^{(0)}z^{(0)}
\end{array} \right.\
$$
\\
$$
 V_{2}:=\left\{
\begin{array}{l}
x'=\alpha (x^{(0)})^2+2x^{(0)}y^{(0)}+2x^{(0)}z^{(0)} \\
y'=\beta(x^{(0)})^2+2y^{(0)}z^{(0)}\\
z'=(y^{(0)})^2 +(z^{(0)})^2
\end{array} \right.\
$$
\\
$$
 V_{3}:=\left\{
\begin{array}{l}
x'=\alpha (x^{(0)})^2+2y^{(0)}z^{(0)}\\
y'=\beta (x^{(0)})^2+2x^{(0)}y^{(0)}+2x^{(0)}z^{(0)}\\
z'=(y^{(0)})^2 +(z^{(0)})^2
\end{array} \right.\
$$
\\
$$ 
V_{4}:=\left\{
\begin{array}{l}
x'=\alpha(x^{(0)})^2  \\
y'=\beta (x^{(0)})^2+2x^{(0)}y^{(0)}+2x^{(0)}z^{(0)}\\
z'=(y^{(0)})^2 +(z^{(0)})^2+2y^{(0)}z^{(0)}
\end{array}\right.\
$$
\\
$$
\quad \quad \quad V_{5}:=\left\{
\begin{array}{l}
x'=\alpha (x^{(0)})^2 +2y^{(0)}z^{(0)}\\ 
y'=\beta (x^{(0)})^2\\
z'=(y^{(0)})^2 +(z^{(0)})^2+2x^{(0)}y^{(0)}+2x^{(0)}z^{(0)}
\end{array} \right.\
$$
\\
$$ 
\quad \quad  V_{6}:=\left\{
\begin{array}{l}
x'=\alpha (x^{(0)})^2  \\
y'=\beta(x^{(0)})^2+2y^{(0)}z^{(0)}\\
z'=(y^{(0)})^2 +(z^{(0)})^2+2x^{(0)}y^{(0)}+2x^{(0)}z^{(0)}
\end{array} \right.\
$$
\\
$$
V_{7}:=\left\{
\begin{array}{l}
x'=\alpha (x^{(0)})^2 +2x^{(0)}y^{(0)}+2x^{(0)}z^{(0)} \\
y'=(y^{(0)})^2 +(z^{(0)})^2\\
z'=\beta (x^{(0)})^2+2y^{(0)}z^{(0)}
\end{array} \right.\
$$
\\
$$
V_{8}:=\left\{
\begin{array}{l}
x'=\alpha (x^{(0)})^2 +2x^{(0)}y^{(0)}+2x^{(0)}z^{(0)}\\
y'=(y^{(0)})^2 +(z^{(0)})^2+2y^{(0)}z^{(0)}\\
z'=\beta (x^{(0)})^2
\end{array} \right.\
$$
\\
$$
\quad \quad \quad V_{9}:=\left\{
\begin{array}{l}
x'=\alpha (x^{(0)})^2+2y^{(0)}z^{(0)}  \\
y'=(y^{(0)})^2 +(z^{(0)})^2+2x^{(0)}y^{(0)}+2x^{(0)}z^{(0)}\\
z'=\beta (x^{(0)})^2
\end{array} \right.\
$$
\\
$$
\quad \quad \quad V_{10}:=\left\{
\begin{array}{l}
x'=\alpha (x^{(0)})^2  \\
y'=(y^{(0)})^2 +(z^{(0)})^2+2x^{(0)}y^{(0)}+2x^{(0)}z^{(0)} \\
z'=\beta (x^{(0)})^2+2y^{(0)}z^{(0)} 
\end{array} \right.\
$$
\\
$$
V_{11}:=\left\{
\begin{array}{l}
x'=\alpha (x^{(0)})^2+2y^{(0)}z^{(0)}   \\
y'=(y^{(0)})^2 +(z^{(0)})^2\\
z'=\beta (x^{(0)})^2+2x^{(0)}y^{(0)}+2x^{(0)}z^{(0)}
\end{array} \right.\
$$
\\
$$
V_{12}:=\left\{
\begin{array}{l}
x'=\alpha (x^{(0)})^2  \\
y'=(y^{(0)})^2 +(z^{(0)})^2 +2y^{(0)}z^{(0)}\\
z'=\beta  (x^{(0)})^2+2x^{(0)}y^{(0)}+2x^{(0)}z^{(0)}
\end{array} \right.\
$$
\\
$$
V_{13}:=\left\{
\begin{array}{l}
x'=  \beta (x^{(0)})^2+2x^{(0)}y^{(0)}+2x^{(0)}z^{(0)}\\
y'=\alpha (x^{(0)})^2\\
z'=(y^{(0)})^2 +(z^{(0)})^2 +2y^{(0)}z^{(0)}
\end{array} \right.\
$$
\\
$$
V_{14}:=\left\{
\begin{array}{l}
x'=  \beta (x^{(0)})^2+2x^{(0)}y^{(0)}+2x^{(0)}z^{(0)}\\
y'=\alpha (x^{(0)})^2+2y^{(0)}z^{(0)}\\
z'=(y^{(0)})^2 +(z^{(0)})^2
\end{array} \right.\
$$
\\
$$
V_{15}:=\left\{
\begin{array}{l}
x'=  \beta (x^{(0)})^2+2y^{(0)}z^{(0)}\\
y'=\alpha (x^{(0)})^2+2x^{(0)}y^{(0)}+2x^{(0)}z^{(0)}\\
z'=(y^{(0)})^2 +(z^{(0)})^2
\end{array} \right.\
$$
\\
$$ 
V_{16}:=\left\{
\begin{array}{l}
x'=  \beta(x^{(0)})^2\\
y'=\alpha  (x^{(0)})^2+2x^{(0)}y^{(0)}+2x^{(0)}z^{(0)}\\
z'=(y^{(0)})^2 +(z^{(0)})^2+2y^{(0)}z^{(0)}
\end{array} \right.\
$$
\\
$$
\quad \quad \quad V_{17}:=\left\{
\begin{array}{l}
x'=  \beta(x^{(0)})^2+2y^{(0)}z^{(0)}\\
y'=\alpha  (x^{(0)})^2\\
z'=(z^{(0)})^2+(y^{(0)})^2+2x^{(0)}y^{(0)}+2x^{(0)}z^{(0)}
\end{array} \right.\
$$
\\
$$ 
\quad \quad \quad  V_{18}:=\left\{
\begin{array}{l}
x'=  \beta (x^{(0)})^2\\
y'=\alpha (x^{(0)})^2+2y^{(0)}z^{(0)}\\
z'=(z^{(0)})^2+(y^{(0)})^2+2x^{(0)}y^{(0)}+2x^{(0)}z^{(0)}
\end{array} \right.\
$$
\\
$$
V_{19}:=\left\{
\begin{array}{l}
x'=  \beta(x^{(0)})^2+2x^{(0)}y^{(0)}+2x^{(0)}z^{(0)}\\
y'=(y^{(0)})^2 +(z^{(0)})^2\\
z'=\alpha (x^{(0)})^2+2y^{(0)}z^{(0)}
\end{array} \right.\
$$
\\
$$
V_{20}:=\left\{
\begin{array}{l}
x'=  \beta (x^{(0)})^2+2x^{(0)}y^{(0)}+2x^{(0)}z^{(0)}\\
y'=(y^{(0)})^2 +(z^{(0)})^2+2y^{(0)}z^{(0)}\\
z'=\alpha (x^{(0)})^2
\end{array} \right.\
$$
\\
$$
\quad \quad \quad V_{21}:=\left\{
\begin{array}{l}
x'=  \beta (x^{(0)})^2+2y^{(0)}z^{(0)}\\
y'=(y^{(0)})^2 +(z^{(0)})^2+2x^{(0)}y^{(0)}+2x^{(0)}z^{(0)}\\
z'=\alpha (x^{(0)})^2
\end{array} \right.\
$$
\\
$$ 
\quad \quad \quad V_{22}:=\left\{
\begin{array}{l}
x'=  \beta (x^{(0)})^2\\
y'=(y^{(0)})^2 +(z^{(0)})^2+2x^{(0)}y^{(0)}+2x^{(0)}z^{(0)}\\
z'=\alpha (x^{(0)})^2+2y^{(0)}z^{(0)}
\end{array} \right.\
$$
\\
$$
V_{23}:=\left\{
\begin{array}{l}
x'=  \beta (x^{(0)})^2+2y^{(0)}z^{(0)}\\
y'=(y^{(0)})^2 +(z^{(0)})^2\\
z'=\alpha (x^{(0)})^2+2x^{(0)}y^{(0)}+2x^{(0)}z^{(0)}
\end{array} \right.\
$$
\\
$$
V_{24}:=\left\{
\begin{array}{l}
x'=  \beta (x^{(0)})^2\\
y'=(y^{(0)})^2 +(z^{(0)})^2+2z^{(0)}y^{(0)}\\
z'=\alpha (x^{(0)})^2+2x^{(0)}y^{(0)}+2x^{(0)}z^{(0)}
\end{array} \right.\
$$
\\
$$
\quad \quad \quad V_{25}:=\left\{
\begin{array}{l}
x'=(y^{(0)})^2 +(z^{(0)})^2+2x^{(0)}y^{(0)}+2x^{(0)}z^{(0)} \\
y'=\alpha (x^{(0)})^2\\
z'=\beta (x^{(0)})^2+2y^{(0)}z^{(0)}
\end{array} \right.\
$$
\\
$$
\quad \quad \quad V_{26}:=\left\{
\begin{array}{l}
x'=(y^{(0)})^2 +(z^{(0)})^2+2x^{(0)}y^{(0)}+2x^{(0)}z^{(0)}  \\
y'=\alpha (x^{(0)})^2+2y^{(0)}z^{(0)}\\
z'=\beta (x^{(0)})^2
\end{array} \right.\
$$
\\
$$
V_{27}:=\left\{
\begin{array}{l}
x'= (y^{(0)})^2 +(z^{(0)})^2+2z^{(0)}y^{(0)}\\
y'=\alpha (x^{(0)})^2+2x^{(0)}y^{(0)}+2x^{(0)}z^{(0)}\\
z'=\beta (x^{(0)})^2
\end{array} \right.\
$$
\\
$$
V_{28}:=\left\{
\begin{array}{l}
x'=(y^{(0)})^2 +(z^{(0)})^2 \\
y'=\alpha (x^{(0)})^2+2x^{(0)}y^{(0)}+2x^{(0)}z^{(0)}\\
z'=\beta  (x^{(0)})^2+2y^{(0)}z^{(0)}
\end{array} \right.\
$$
\\
$$
V_{29}:=\left\{
\begin{array}{l}
x'= (y^{(0)})^2 +(z^{(0)})^2+2z^{(0)}y^{(0)}\\
y'=\alpha (x^{(0)})^2\\
z'=\beta (x^{(0)})^2+2x^{(0)}y^{(0)}+2x^{(0)}z^{(0)}
\end{array} \right.\
$$
\\
$$
V_{30}:=\left\{
\begin{array}{l}
x'=(y^{(0)})^2 +(z^{(0)})^2 \\
y'=\alpha  (x^{(0)})^2+2y^{(0)}z^{(0)}\\
z'=\beta (x^{(0)})^2+2x^{(0)}y^{(0)}+2x^{(0)}z^{(0)}
\end{array} \right.\
$$
\\
$$
\quad \quad \quad V_{31}:=\left\{
\begin{array}{l}
x'= (y^{(0)})^2 +(z^{(0)})^2+2x^{(0)}y^{(0)}+2x^{(0)}z^{(0)}\\
y'=\beta (x^{(0)})^2 \\
z'=\alpha(x^{(0)})^2+2y^{(0)}z^{(0)}
\end{array} \right.\
$$
\\
$$
 \quad \quad \quad V_{32}:=\left\{
\begin{array}{l}
x'=(y^{(0)})^2 +(z^{(0)})^2+2x^{(0)}y^{(0)}+2x^{(0)}z^{(0)} \\
y'=\beta (x^{(0)})^2+2y^{(0)}z^{(0)}\\
z'=\alpha (x^{(0)})^2 
\end{array} \right.\
$$
\\
$$
V_{33}:=\left\{
\begin{array}{l}
x'=(y^{(0)})^2 +(z^{(0)})^2 +2z^{(0)}y^{(0)} \\
y'=\beta (x^{(0)})^2+2x^{(0)}y^{(0)}+2x^{(0)}z^{(0)}\\
z'=\alpha (x^{(0)})^2
\end{array} \right.\
$$
\\
$$
V_{34}:=\left\{
\begin{array}{l}
x'=(y^{(0)})^2 +(z^{(0)})^2  \\
y'=\beta (x^{(0)})^2+2x^{(0)}y^{(0)}+2x^{(0)}z^{(0)}\\
z'=\alpha(x^{(0)})^2+2y^{(0)}z^{(0)}
\end{array} \right.\
$$
\\
$$
V_{35}:=\left\{
\begin{array}{l}
x'=(y^{(0)})^2 +(z^{(0)})^2 +2z^{(0)}y^{(0)} \\
y'=\beta (x^{(0)})^2\\
z'=\alpha (x^{(0)})^2+2x^{(0)}y^{(0)}+2x^{(0)}z^{(0)}
\end{array} \right.\
$$
\\
$$
V_{36}:=\left\{
\begin{array}{l}
x'=(y^{(0)})^2 +(z^{(0)})^2  \\
y'=\beta (x^{(0)})^2+2y^{(0)}z^{(0)}\\
z'=\alpha (x^{(0)})^2+2x^{(0)}y^{(0)}+2x^{(0)}z^{(0)}
\end{array} \right.\
$$
\\

Evidently,  class $\xi^{(as)}$-QSO contains 36 operators, which operators are too numerous to  explore individually
. Therefore, we classify such operators  into small classes and examine only the operators within these classes.

\begin{theorem}\label{500} Let $ \left\lbrace V_{1},\cdots,V_{36} \right\rbrace $ be the $\xi^{(as)}$ -QSO presented  above. Then, these operators are divided into  18 non-isomorphic classes:

\begin{tabular}{c c c c c c}\\
$\quad G_{1}=\{ V_{1},V_{8}\}$, & $\quad G_{2}=\{ V_{2},V_{7}\}$, & 
$\quad G_{3}=\{ V_{3},V_{11}\}$,& 
$\quad G_{4}=\{ V_{4},V_{12}\}$,\\
\\
$\quad G_{5}=\{ V_{5}, V_{9}\}$,&
$\quad G_{6}=\{ V_{6},V_{10}\}$,&
$\quad G_{7}=\{ V_{13},V_{20}\}$, & $\quad G_{8}=\{ V_{14},V_{19}\}$, &\\
\\
 $\quad \quad G_{9}=\{ V_{15},V_{23}\}$,& 
$\quad G_{10}=\{ V_{16},V_{24}\}$,&
$\quad G_{11}=\{ V_{17},V_{21}\}$,&
$\quad G_{12}=\{ V_{18},V_{22}\}$,&\\
\\
$\quad \quad G_{13}=\{ V_{25},V_{32}\}$, & $\quad G_{14}=\{ V_{26},V_{31}\}$, & 
$\quad G_{15}=\{ V_{27},V_{35}\}$,&
$\quad G_{16}=\{ V_{28},V_{36}\}$,\\ 
\\
$ \quad\quad G_{17}=\{ V_{29},V_{33}\}$,&
$\quad G_{18}=\{ V_{30},V_{34}\}$.&
\end{tabular}

\end{theorem}

 \begin{proof} Evidently, the partitions $\xi_2  $ of $ \mathbf{P}_3 $ and $\xi_3  $ of $ \Delta_3$ are invariant only under  the permutation $\pi=\begin{pmatrix}
x^{(0)} & y^{(0)} & z^{(0)}\\
x^{(0)} & z^{(0)} & y^{(0)}
\end{pmatrix}$. Therefore,  the given operators  should be classified with respect to the remuneration of their coordinates. Consequently,  we have to perform $ \pi V \pi^{-1} $ transformation on all the operators.\\

Starting with $ V_{1} $ as the first operator, we obtain\\

$V_{1}\left( \pi^{-1}(x^{(0)},y^{(0)},z^{(0)})\right) = V_{1}\left(x^{(0)},z^{(0)},y^{(0)}\right) =(  \alpha (x^{(0)})^2+2x^{(0)}y^{(0)}+2x^{(0)}z^{(0)},\beta(x^{(0)})^2,$\\

$(y^{(0)})^2+(z^{(0)})^2+2y^{(0)}z^{(0)}) . $
Thus,  
$$ \pi V_{1} \pi^{-1} =\left(\alpha (x^{(0)})^2+2x^{(0)}y^{(0)}+2x^{(0)}z^{(0)},(y^{(0)})^2+(z^{(0)})^2+2y^{(0)}z^{(0)} , (a-1)\beta (x^{(0)})^2 \right) = V_{8} .$$

We can derive the other classes by following the same procedure. This process completes the proof.

 \end{proof}
\section{Dynamics of classes $G_{3}$ and $ G_{9}$}

This section explores the dynamics of   $\xi^{(as)}$-QSO  $V_{3,15}:S^{2}\rightarrow S^{2} $  
selected from $G_{3}  $ and $G_{9}  $. To begin,  $V_{3}$ is rewritten as follows:

 \begin{equation}\label{600} 
V_{3}:=\left\{
\begin{array}{l}
 x'=  \alpha  (x^{(0)})^{2}+2y^{(0)}z^{(0)}   \\
y'=\left( 1-\alpha\right)  (x^{(0)})^{2}+2x^{(0)}\left( 1-x^{(0)}\right) \\
z'= (z^{(0)})^{2}+ (y^{(0)})^{2}
\end{array} \right.
\end{equation}
The operator $V_{3}$ can be redrafted  as a convex combination $ V_{3}= \alpha  W_{1}+\left( 1-\alpha\right) W_{2} $,\\
   where \\
 \begin{equation}\label{600}
W_{1}:=\left\{
\begin{array}{l}
x'= (x^{(0)})^{2}+2y^{(0)}z^{(0)}  \\
y'=2x^{(0)}\left( 1-x^{(0)}\right)\\
z'= (z^{(0)})^{2}+ (y^{(0)})^{2}
\end{array} \right.
\end{equation}\\
and
 \begin{equation}\label{601}
W_{2}:=\left\{
\begin{array}{l}
x'=2y^{(0)}z^{(0)}  \\
y'=2x^{(0)}-(x^{(0)})^{2}\\
z'=(z^{(0)})^{2}+ (y^{(0)})^{2}
\end{array} \right.
\end{equation} 
 
\begin{theorem}\label{602}  Let   $  W_{1}:S^{2}\rightarrow S^{2} $  be a $\xi^{(as)}$-QSO  given by \eqref{600}and  $ x_{1}^{(0)}=(x^{(0)},y^{(0)},z^{(0)})\notin Fix(W_{1})$  be any  an initial point in the simplex $ S^{2} $. Then, the following    statements are true:
\begin{enumerate}
\item[(i)]$ Fix(W_{1})=\left\{e_{1},e_{3},(\dfrac{3-\sqrt{3}}{4},\dfrac{\sqrt{3}}{4},\frac{1}{4})\right\rbrace  $,
\item[(ii)] $ \omega _{W_{1}}(x_{1}^{(0)} )=\left\lbrace (\dfrac{3-\sqrt{3}}{4},\dfrac{\sqrt{3}}{4},\frac{1}{4})\right\rbrace  $.
\end{enumerate}
\end{theorem}
\begin{proof}
Let $  W_{1}:S^{2}\rightarrow S^{2} $ be a $\xi^{(as)}$-QSO  given by \eqref{600},   $ x_{1}^{(0)}\notin Fix( W_{1})$  be any   initial point in  simplex $ S^{2} $, and   $ \left\lbrace W_{1}^{(n)} \right\rbrace _{n=1}^{\infty}  $
 be a trajectory of $ W_{1} $  starting from  point $x_{1}^{(0)}  $.\\
 
(i) The set of  fixed points of $ W_{1} $ are obtained by finding the solution for the   following  system of equations:
\begin{equation}\label{602ytut}
\left\{
\begin{array}{l}
x=x^{2}+2yz \\
y=2x\left( 1-x\right)\\
z=y^{2}+z^{2}
\end{array} \right.
\end{equation}

 By depending on the first equation in system \eqref{602ytut}, we derive $ x-x^{2}=2yz $. Subsequently, the last equation is multiplied by 2, and the new equation is substituted into the second equation in system \eqref{602ytut}. We obtain  $ y(1-4z)=0 $ and  find $ y=0 $ or $ z=\frac{1}{4} $. If $ y=0 $, then $ x=0$ or $ x=1$ can be easily found; hence, the fixed points are $ e_{1}=(1,0,0) $ and $ e_{3}=(0,0,1) $. If $ z=\frac{1}{4} $,  then $ y=\dfrac{\sqrt{3}}{4} $ and  $x=\dfrac{3-\sqrt{3}}{4}$ can be found by using the first and third equation  in system  \eqref{602ytut}. Therefore, the fixed point is $ \left( \dfrac{3-\sqrt{3}}{4},\dfrac{\sqrt{3}}{4},\frac{1}{4}\right) $.\\
\\
\\
(ii) To investigate the dynamics of $ W_{1} $,   the following regions are introduced:\\
\begin{eqnarray*}
A_1:&=&\{x_{1}^{(0)} \in S^{2}:  0\leq x^{(0)},y^{(0)},z^{(0)}\leq \frac{1}{2} \},\\
A_2:&=&\{x_{1}^{(0)} \in S^{2}:   \frac{1}{2}<  x^{(0)}<1 \},\\
A_3:&=&\{x_{1}^{(0)} \in S^{2}:   \frac{1}{2}<  y^{(0)}<1  \},\\
A_4:&=&\{x_{1}^{(0)} \in S^{2}:   \frac{1}{2}< z^{(0)}<1  \},\\
A_5:&=&\{x_{1}^{(0)} \in S^{2}:  0< z^{(0)} \leq x^{(0)}< y^{(0)}< \frac{1}{2} \},\\
A_6:&=&\{x_{1}^{(0)} \in S^{2}:   0\leq y^{(0)} \leq z^{(0)}\leq x^{(0)}\leq  \frac{1}{2}\},\\
A_7:&=&\{x_{1}^{(0)} \in S^{2}:   0\leq x^{(0)} \leq y^{(0)}\leq z^{(0)}\leq  \frac{1}{2}\},\\
A_8:&=&\{x_{1}^{(0)} \in S^{2}:  0< z^{(0)} \leq x^{(0)}\leq \dfrac{1}{3},\quad \dfrac{1}{3} < y^{(0)}\leq \frac{1}{2} \}.
\end{eqnarray*}
  
 Subsequently,  the behavior of $ W_{1} $ across all the aforementioned regions is explored. Then,  the behavior of $W_{1}  $ will be described. To achieve this objective, the following results should be shown:\\

(1) Let $x_{1}^{(0)}\in A_1   $. Then,  $0\leq x^{(0)},y^{(0)},z^{(0)}\leq \frac{1}{2} $. Evidently $-1 \leq 3x^{(0)} - 1 \leq \frac{1}{2}  $ by squaring and adding $ -3(y^{(0)}-z^{(0)})^{2} $. The last inequality becomes   $0\leq (3x^{(0)} - 1)^{2}-3(y^{(0)}-z^{(0)})^{2} \leq 1  $, and $9(x^{(0)})^{2}-6x^{(0)}+1-3(y^{(0)}-z^{(0)})^{2} \leq 1   $ is obtained.    Dividing the previous inequality  by three after adding two to   both parts  of the inequality will derive  $$3(x^{(0)})^{2} -2x^{(0)}+1-(y^{(0)}-z^{(0)})^{2}\leq 1 .$$ Therefore,  $$2(x^{(0)})^{2} +(y^{(0)}+z^{(0)})^{2}-(y^{(0)}-z^{(0)})^{2}\leq 1 .$$Then,  $2(x^{(0)})^{2}+4y^{(0)}z^{(0)}\leq 1  $,  which implies that $ x'\leq \frac{1}{2} $.  To show that  $y'\leq \frac{1}{2}  $, one can check that $ y'\leq \frac{1}{2}$        $ \forall   x_{1}^{(0)}  $. Evidently see that $ 0\leq (y^{(0)})^{2} ,(z^{(0)})^{2}\leq  \frac{1}{4}   $,  which  implies that $z'\leq \frac{1}{2}  $. Hence, $ A_1 $ is an invariant region.\\

(2) The second coordinate  of $W_{1}  $ is less than $ \frac{1}{2} $ at any initial point $ x_{1}^{(0)}  $, thereby indicating that $ A_3 $ is not an invariant region. Then, we intend to show that $  A_2 $ is also not an invariant region. To achieve this objective,  we suppose that $  A_2 $ is an invariant region, which indicates that $y'\leq x'  $ and $z'\leq x' $. However,  $$x'= (x^{(0)})^{2}+2y^{(0)}z^{(0)} \leq (x^{(0)})^{2}+(y^{(0)})^{2}+(z^{(0)})^{2}\leq x^{(0)}(x^{(0)}+y^{(0)}+z^{(0)})=x^{(0)}. $$ Then $ \dfrac{x'}{x^{(0)}}<1 $, which implies that the first coordinate is a decreasing bounded sequence that converges to zero, thereby contradicting our assumption. Hence, if $ x_{1}^{(0)}\in  A_2 \cup A_3  $, then $ n_{k_{1}},  n_{k_{2}}  \in \mathbb{N} $, such that the sequences $ x^{(n_{k_{1}})} $ and $ y^{(n_{k_{2}})} $ tend toward the invariant region $ A_1 $.\\

(3) Thereafter, we intend to show that if $ x_{1}^{(0)}\in  A_4 $, then  $ n_{k} \in \mathbb{N} $, such that the sequence $ z^{(n_{k})} $  returns to  region $ A_1 $. To achieve this objective,  $  A_{4}$ is supposed as an invariant region; hence, $z'\geq y'+x'  $ and $ x' ,y' \leq \frac{1}{2}  $. Evidently $x' \leq y'   $. By using the last inequality and the first coordinate of $W_{1}  $,  we obtain $y' \geq 2y^{(0)}z^{(0)}  $. That is, $ z' \leq \frac{1}{2} $, which repudiates our assumption. Hence,  region $A_4  $ is not  invariant.\\   

(4) Given that $ y',x'\leq \dfrac{1}{2} $, we can easily conclude    $x'\leq y'  $ thereby indicating $A_6  $ is impossible to be  invariant region. Subsequently, we intend to verify whether   $A_7  $ is an  invariant region. Let $x_{1}^{(0)} \in A_7  $. Then,

 $$\quad z'=(z^{(0)})^{2}+(y^{(0)})^{2}\leq (x^{(0)})^{2}+(y^{(0)})^{2}+(z^{(0)})^{2}$$  

  $\quad  \quad \quad  \quad \quad  \quad \quad  \quad\quad  \quad\quad  \quad\quad  \quad \quad  \quad \quad  \quad\quad  $ 
 $\leq  x^{(0)}z^{(0)}+y^{(0)}z^{(0)}+(z^{(0)})^{2}$\\ 
 
$\quad  \quad \quad  \quad \quad  \quad \quad  \quad\quad  \quad\quad  \quad\quad  \quad \quad  \quad \quad  \quad\quad  $  $=z^{(0)}(x^{(0)}+y^{(0)}+z^{(0)})=z^{(0)}$.
\\

We determine that $ \frac{z'}{z^{(0)}}< 1 $, which indicates that  $z^{(n)}  $ is a decreasing bounded sequence, i.e., $z^{(n)}  $ converges to the fixed point zero,  thereby negating our presumption. Thus,   region  $A_7  $ is not  invariant. Then, we  consider  a new sequence $x'+z'= 2(x^{(0)})^{2}-2x^{(0)}+1  $. The new sequence has a minimum value of $ \dfrac{1}{2} $, which indicates that all coordinates are  greater than zero and less than $ \dfrac{1}{2} $. Hence, if $ x_{1}^{(0)}\in  A_6\cup A_7 \cup A_1  $, then,  $ n_{k_{1}}, n_{k_{2}}, n_{k_{3}} \in \mathbb{N} $, such that the sequences $ x^{(n_{k_{1}})} $, $ y^{(n_{k_{2}})} $, and $ z^{(n_{k_{3}})} $  return to   invariant region $ A_5 $.\\

(5) Let $ x^{(0)}\leq\frac{1}{3} $.  Whether the maximum  value of the first  coordinate  $ x'=(x^{(0)})^{2}+2y^{(0)}(1-x^{(0)}-y^{(0)})$ occurs when $\left( \frac{1}{3},\frac{1}{3},\frac{1}{3}\right) $ can be easily checked. Thus, $ x^{(n)}\leq \frac{1}{3} $ and $ z^{(n)}\leq \frac{1}{3} $.  Given that all coordinates are equal to one, we conclude that  $ y^{(n)}\geq \frac{1}{3} $. Therefore, if $x_{1}^{(0)} \in A_{5}$, then $ n_{k} \in \mathbb{N} $, such that $W^{(n_{k})}_{1}  $ returns to $ A_{8} $.  Hence, $ A_{8} $ is an invariant region.   \\

We have proven that if $x_{1}^{(0)} \in A_{i}$, $ i \in \lbrace1,\dots,7\rbrace $, then the trajectory $ \lbrace W_{1}^{(n)} \rbrace _{n=1}^{\infty}  $  goes to invariant region $ A_{8} $. Thus, exploring the dynamics of $W_{1}  $ over region $ A_{8} $ is adequate.
Evidently, $ y^{(n)} $ is a bounded increasing   sequence. Given that $y^{(n)}+x^{(n)} $ is a bounded decreasing  sequence and $x^{(n)}=y^{(n)}-y^{(n)}+x^{(n)}$, we conclude  that  $  x^{(n)} $ is a decreasing bounded  sequence that converges to $\dfrac{3-\sqrt{3}}{4}  $. Thus, we have $ y^{(n)} $ converging to $ \dfrac{\sqrt{3}}{4} $. Therefore, $\omega _{W_{1}}(x_{1}^{(0)})= \left\lbrace (\dfrac{3-\sqrt{3}}{4},\dfrac{\sqrt{3}}{4},\frac{1}{4})\right\rbrace$, which is the desired conclusion.
\end{proof}\\ 
\begin{theorem}\label{603} Let   $  W_{2}:S^{2}\rightarrow S^{2} $  be a $\xi^{(as)}$-QSO  given by \eqref{601}and  $ x_{1}^{(0)}=(x^{(0)},y^{(0)},z^{(0)})\notin Fix(W_{2})$  be any   initial point in  simplex $ S^{2} $. Then, the following    statements are true:\\
\begin{enumerate}
\item[(i)]$ Fix(W_{2})=\left\{e_{3},(x^{\centerdot},y^{\centerdot},z^{\centerdot})\right\rbrace  $,\\

where $x^{\centerdot}=\frac{-1}{6}\sqrt[3]{t}-\frac{8}{3\sqrt[3]{t}+\frac{5}{3}}$, $y^{\centerdot}=\frac{-1}{6}\frac{3\sqrt{17}\sqrt[3]{t}+2\sqrt[3]{t^{2}}-24\sqrt{17}-5\sqrt[3]{t}-88}{\sqrt[3]{t^{2}}}  $, 
 $z^{\centerdot}=\frac{-1}{6}\frac{2\sqrt[3]{t^{2}}-3\sqrt{17}\sqrt[3]{t}-11\sqrt[3]{t},+6\sqrt{17}-10}{\sqrt[3]{t^{2}}}  $,
 and $ t=\left( 98+18\sqrt{17}\right)  $.
 
\item[(ii)] \begin{equation}\label{fssddsa}
Per_{2}(W_{2})=\left\{
\begin{array}{l}
\begin{split}
 e_{3},(0,y^{\circ},1-y^{\circ}) \end{split}, if \ \ x^{(0)}=0  \ \\
 \\
\begin{split}
e_{3},(x^{\circ},0,1-x^{\circ})
\end{split}, if \ \ y^{(0)}=0  \\ 
\end{array} \right. \\
\end{equation}

 where $y^{\circ}=\dfrac{1}{6}(1+3\sqrt{57})^{\frac{1}{3}} -\frac{4}{3(1+3\sqrt{57})^{\frac{1}{3}}}+\frac{2}{3}$, $x^{\circ}=\dfrac{-1}{6}(46+6\sqrt{57})^{\frac{1}{3}}-\frac{2}{3(46+6\sqrt{57})^{\frac{1}{3}}}+\frac{4}{3}  $.
 
\item[(iii)] \begin{equation}\label{fssddsa}
\omega _{w_{2}}(x_{1}^{(0)} )=\left\{
\begin{array}{l}
\begin{split}
 \quad\quad (x^{\centerdot},y^{\centerdot},z^{\centerdot}) \end{split}\ \ \quad\quad\quad\quad \quad\quad, if \ \ x_{1}^{(0)}\in int\left( S^{2}\right)   \ \\
 \\
\begin{split}
(x^{\circ},0,1-x^{\circ}),(0,y^{\circ},1-y^{\circ})
\end{split}\ \ , if \ x_{1}^{(0)} \in \overline{int\left( S^{2}\right)}   \\
\\
 \begin{split}
\quad\quad \quad e_{3} \end{split}\ \ \quad \quad\quad\quad\quad\textbf{}\quad\quad \quad, if \ \ x^{(0)}, y^{(0)}=1   \ \\
 \\ 
\end{array} \right.
\end{equation}

\end{enumerate}
\end{theorem}
\begin{proof}Let $  W_{2}:S^{2}\rightarrow S^{2} $ be a $\xi^{(as)}$-QSO  given by \eqref{601}, $ x_{1}^{(0)}\notin Fix( W_{2})$  be any   initial point in $ S^{2} $, and   $ \left\lbrace W_{2}^{(n)} \right\rbrace _{n=1}^{\infty}  $
 be a trajectory of $ W_{2} $  starting from  point $x_{1}^{(0)} $.\\

(i) The set of  fixed points of $ W_{2} $ is obtained by finding the solution for the   following system of equations:\\
 \begin{equation}\label{604}
\left\{
\begin{array}{l}
x= 2yz \\
y=2x-x^{2}\\
z=z^{2}+y^{2}
\end{array} \right.
\end{equation}\\
On the basis of the first equation in system \eqref{604}, we have $ z=\frac{x}{2y} $. By using $ z=1-y-x $ and the second equation in system \eqref{604}, we obtain $ 3x-14x^{2}+10x^{3}-2x^{4}=0 $. Thus, the roots of the previous  equation are $ \left\lbrace0,x^{\centerdot} \right\rbrace  $. By compensating for the values of $ x $, namely,   $ x=0 $ and $ x=x^{\centerdot} $ in the second equation in system \eqref{604}, we obtain $ y=0$ and $z=1$ or $y=y^{\centerdot} $ and $z=z^{\centerdot} $. Therefore, the fixed points of $W_{2} $ are $ e_{3}=(0,0,1)$ and $(x^{\centerdot},y^{\centerdot},z^{\centerdot})$.
\\
\\
(ii) To find $ 2-$periodic points of $ W_{2} $, we should prove that $ W_{2} $ has no any order periodic points in  set  $ S^{2}\setminus L_{1}\cup L_{2} $, where $ L_{1}=\lbrace x_{1}^{(0)} \in S^{2}: x^{(0)}=0  \rbrace  $ and $ L_{2}=\lbrace x_{1}^{(0)} \in S^{2}: y^{(0)}=0  \rbrace  $. Evidently, the second coordinate of $ W_{2} $ increases along the iteration of $ W_{2} $ in  set $S^{2}\setminus L_{2} $. Consider a new sequence  $ x'+y'=2x^{(0)}-(x^{(0)})^{2}+2y^{(0)}(1-x^{(0)}-y^{(0)}) $. Whether $x'+y'  $ is a decreasing sequence can be easily checked, thereby indicating  that  sequence $x^{(n)} $ is decreasing because $x^{(n)}=y^{(n)}-y^{(n)}+x^{(n)} $. Thus, the first coordinate of $ W_{2} $ decreases along the iteration of $ W_{2} $ in  set $\ S^{2}\setminus L_{1} $, which indicates that $ W_{2} $ has no any order $ 2- $periodic points in  set $ S^{2}\setminus  L_{1}\cup L_{2} $. Therefore, finding $ 2- $periodic points of $ W_{2} $ in  $ L_{1}\cup L_{2} $ is sufficient. To find  $ 2-$periodic points, the succeeding  system of equations should be solved:
\begin{equation}\label{602}
\left\{
\begin{array}{l}
x=2(2x-x^{2})(y^{2}+z^{2}) \\
y=4yz-4y^{2}z^{2}\\
z=(2x-x^{2})(y^{2}+z^{2})^{2}
\end{array} \right.
\end{equation} \\
 First, we  start when $ x=0 $. Then, we find the solution for $ y=4y-8y^{2}+8y^{3}-4y^{4} $. We obtain the following solution: $ y=0 $ or $ y=y^{\circ} $. If  $ y=0 $, then $ z=1 $.   If $ y=y^{\circ} $, then $ z=1-y^{\circ} $. Therefore, $ e_{3} $ and $ (0,y^{\circ},1-y^{\circ}) $  are $ 2- $periodic points. On the other hand, if  $ y=0 $, then  the  solutions for the following equation: $ x=2(2-x-x^{2})(1-x)^{2} $ are $ x=0 $ or $ x=x^{\circ} $. If  $ x=0 $, then $ z=0 $;  if $ x=x^{\circ} $, then $ z=1-x^{\circ} $. Therefore, $ e_{3} $ and $ (x^{\circ},0,1-x^{\circ}) $ are $ 2- $periodic points.\\
 
(iii) To investigate the dynamics of $ W_{2} $,   the following regions are introduced:\\
\begin{eqnarray*} 
 \ell_1:&=&\{x_{1}^{(0)}\in  int\left(S^{2} \right):   0	< x^{(0)},y^{(0)},z^{(0)}\leq \frac{1}{2} \};\\
 \ell_2:&=&\{x_{1}^{(0)} \in  int\left(S^{2} \right):  0	< x^{(0)} \leq z^{(0)}\leq y^{(0)}\leq \frac{1}{2} \}.
 \end{eqnarray*}\\
 
 Let $ x_{1}^{(0)}\notin Fix(W_{2})\cup Per_{2}(W_{2}) $ and $ x_{1}^{(0)} \in int\left( S^{2}\right)$ be the initial points where $int\left( S^{2}\right)=\lbrace x_{1}^{(0)} \in S^2 :x^{(0)}y^{(0)}z^{(0)}>0 \rbrace   $. Evidently,  $ y'=2x^{(0)}-(x^{(0)})^{2}\geq x^{(0)} $ and  $ x'=2y^{(0)}z^{(0)}\leq (y^{(0)})^{2}+(z^{(0)})^{2}=z'$,  which indicates  that $x^{(n)}\leq z^{(n)}  $ and
  $x^{(n)}\leq y^{(n)}  $. Subsequently, we are going to prove that $\ell_1  $ is an invariant region. To achieve this objective. we start with $ y^{(n)} $. Suppose that $ y'\geq \frac{1}{2} $ by using the  first coordinate of $ W_{2} $. Then we have $ x'=2y^{(0)}z^{(0)} $, which implies that $ x'\geq z' $. This relation is a contradiction because $x'\leq z'$. Thus, $y^{(n)}\leq\frac{1}{2} $. By performing the same process  used to prove $ y^{(n)}\leq \frac{1}{2} $, we prove that $ z'\leq \frac{1}{2} $. Suppose that $ z'\geq \frac{1}{2} $. By using  the first coordinate in $ W_{2} $, we obtain $ x'\geq y' $, which is another contradiction. Therefore, $\ell_1 $ is an invariant region. Moreover, if $ x_{1}^{(0)} \notin Fix(W_{2})\cup Per_{2}(W_{2}) $, $ x_{1}^{(0)} \in int\left( S^{2}\right)$, and $x_{1}^{(0)}\in \overline{\ell_1}  $, then $n_{k} \in \mathbb{N}$, such that $ W_{2}^{(n_{k})} $  returns  to  invariant region $\ell_1 $. Let us complete proving  that $ \ell_2 $ is an invariant region. To achieve this objective, suppose that $ y'\leq z' $, which indicates that $ z'=(z^{(0)})^{2}+(y^{(0)})^{2}\leq 2(z^{(0)})^{2} $. Then,  $ \dfrac{z'}{z^{(0)}}\leq 1 $. Therefore, $ z^{(n)} $ is a decreasing   bounded sequence. That is $ z^{(n)} $ converges to the fixed point zero. Moreover, $ y^{(n)} $ is an  increasing bounded sequence. Thus, $ y^{(n)} $ converges  to zero. Whether  $ y^{(n)} $  converges  to zero if  $ x^{(n)} $  converges  to zero can be checked. The result implies that  the limiting point for $ W_{2}$  is empty, which is a contradiction. Thus,  $   n_{k} \in \mathbb{N}$, such that  $ z^{(n_{k})} $  returns to  invariant region $ z'\leq y' $, which proves  that $ \ell_2 $ is an invariant region. Moreover, if $ x_{1}^{(0)} \in \ell_{1} $, then  $n_{k} \in \mathbb{N}$, such that $ W_{2}^{(n_{k})} $  returns  to  invariant region $\ell_2 $.\\ 
 
Accordingly, the behavior of $W_{2}  $ can be described.
As  discussed in  proof  part $ 2$ of this theorem, we determine that the first and second coordinates, namely,   $ x^{(n)} $ and $ y^{(n)} $,  are decreasing and increasing sequences respectively. Thus, $ x^{(n)} $ and $ y^{(n)} $ converge to certain fixed point. The first and second coordinates of $ W_{2} $ are converging; thus, the third coordinate also converges. Between the two fixed points, the aforementioned properties of $ W_{2} $ are only satisfied by point  $ (x^{\centerdot},y^{\centerdot},z^{\centerdot}) $. Therefore, the limiting point is  $ \omega_{ W_{2}}( x_{1}^{(0)})= (x^{\centerdot},y^{\centerdot},z^{\centerdot})   $ $ \forall x_{1}^{(0)} \in int\left( S^{2}\right) $.\\ 

To explore the behavior of $ W_{2} $ when $x_{1}^{(0)} \in \overline{int\left( S^{2}\right)}    $, where $ \overline{int\left( S^{2}\right)}=\lbrace x_{1}^{(0)} \in S^{2}: x^{(0)}y^{(0)}z^{(0)}=0 \rbrace   $, consider three cases i.e., when $ x^{(0)}=0 $, $y^{(0)}=0 $, and $  z^{(0)}=0 $. If $  x^{(0)}=0 $, then $ V^{(1)}((0,y^{(0)},z^{(0)}))=(x',0,1-x') $ and $ V^{(2)}((0,y^{(0)},z^{(0)}))=(0,y',1-y') $. By applying this process to the next iteration, we determine that  $V^{(2 n+1)}((0,y^{(0)},z^{(0)}))   =(x^{(2 n+1)},0,1-x^{(2 n+1)})  $ and $V^{(2 n)}((0,y^{(0)},z^{(0)}))   =(0,y^{(2 n)},1-y^{(2 n)})  $. That is, the behavior of $ W_{2} $ in this case will be on the $ xz- $ plane if $ n $ is an odd iteration and on the $ yz- $ plane if $ n $ is an even iteration. When the preceding process is performed when $ y^{(0)}=0 $, we find that $V^{(2 n+1)}((x^{(0)},0,z^{(0)}))   =(0,y^{(2 n+1)}, 1-y^{(2 n+1)}) $  and $  V^{(2 n)}((x^{(0)},0,z^{(0)}))=(x^{(2n)},0,1-x^{(2n)}) $. That is, the behavior of $ W_{2} $ in this case will be on the $ yz- $ plane if $ n $ is an odd iteration and on the $ xz- $ plane if $ n $ is an even iteration. Through the same   process, we determine that $V^{(2 n+1)}((x^{(0)},0,z^{(0)}))   =(0,y^{(2 n+1)}, 1-y^{(2 n+1)}) $  and $  V^{(2 n)}((x^{(0)},0,z^{(0)}))=(x^{(2n)},0,1-x^{(2n)}) $  when $ z^{(0)}=0 $, which indicates that $   n_{k} \in \mathbb{N}$, such that the behavior of $ W_{2} $ when $ z^{(0)}=0 $ case will be on the $ yz- $ plane if $ n $ is an odd iteration and on the $ xz- $ plane if $ n $ is an even iteration. Therefore, studying two cases when $  x^{(0)}=0 $ and $  y^{(0)}=0 $ are sufficient. Starting with $ x^{(0)}=0 $, consider the following function:\\
\begin{equation}\label{try604yjftju}
y^{(2)}=  \nu (y^{(0)})=4y^{(0)}-8(y^{(0)})^2+8(y^{(0)})^3-4(y^{(0)})^4,
\end{equation}  where $ y^{(0)} \in  (0,1) $. $ Fix(  \nu ) \cap (0,1)=\lbrace y^{\circ}\rbrace$ can be shown. Through simple calculations,     $  \nu \left((0,\frac{1}{2}]  \right)\subseteq  [\frac{1}{2},1) $ can be found. Thus, we conclude that $[\frac{1}{2},1) $ is sufficient to study the dynamics of $  \nu  $ at interval $ (0,1) $.\\

To study the behavior of $   \nu $,   interval $ [\frac{1}{2},1) $ is divided into three intervals as follows: $ I_{1}=[\frac{1}{2},y^{\circ}] $, $I_{2}= [y^{\circ},\frac{1}{2}+\frac{1}{2}\sqrt{\sqrt{2}-1}] $, and $I_{3}= [\frac{1}{2}+\frac{1}{2}\sqrt{\sqrt{2}-1},1)  $. Evidently,  $   \nu (  \nu (y^{(0)}))\geq y^{(0)} $ when $ y^{(0)}\in I_{1} $ and $   \nu (  \nu (y^{(0)}))\leq y^{(0)} $ when $ y^{(0)}\in I_{2} $. Therefore, two cases should be discussed separately. \\  
 \begin{enumerate}
 \item[(a)]For any $ n \in\mathbb N $,  $   \nu ^{(2n+2)}(y^{(0)}) \geq   \nu ^{(2n)}(y^{(0)})$ $\forall  y^{(0)}\in I_{1}  $ can be easily  shown. Thus,  $   \nu ^{(2n)}\left( y^{(0)}\right)  $ is an increasing bounded sequence. Furthermore, $   \nu ^{(2n)}(y^{(0)}) $  converges  to a fixed point of $   \nu ^{(2)} $. $ y^{\circ} $   is  also a fixed point of $   \nu ^{(2)}   $, and it is the only possible point of the convergence trajectory. Hence,  sequence $ y^{(2n)} $ converges to $ y^{\circ} $.  \\ \item[(b)] Similarly,   $   \nu ^{(2n+2)}(y^{(0)})\leq     \nu ^{(2n)}(y^{(0)})$ $\forall  y^{(0)}\in I_{2}  $. Thus, 
  $   \nu ^{(2n)} $ a decreasing bounded sequence.  Furthermore, $   \nu ^{(2n)}(y^{(0)}) $ converges to a fixed point of $   \nu ^{(2)} $ .  $ y^{\circ} $   is  also a fixed point of $   \nu ^{(2)}   $, and it is the only possible point of the convergence trajectory. Hence,  sequence $ y^{(2n)} $ converges to $ y^{\circ} $. 
\end{enumerate}
To explore the behavior of $  \nu   $, when $ y^{(0)}\in I_{3} $,  the following claim is required:\\
\begin{claim}
Let $ y^{(0)}\in I_{3} $ . Then, $ n_{k}\in\mathbb N $, such that $   \nu ^{( n_{k})} \in I_{1}\cup I_{2} $.
\end{claim}
\begin{proof}
Let  $ y^{(0)}\in I_{3} $. Suppose that the interval $ I_{3} $ is an invariant interval, which indicates that $ y^{(n)}\in I_{3} $ for any $ n\in\mathbb N $. Evidently, $   \nu  ^{(n+1)}(y^{(0)})\leq    \nu  ^{(n)}(y^{(0)}) $, which results in  $   \nu ^{(n)} $ being a  decreasing  bounded sequence and converging to a fixed point of $  \nu   $. However, $ Fix(  \nu  )\cap I_{3} = \o $, which is a contradiction. Hence,  $ n_{k}\in\mathbb N $,  such that $   \nu ^{( n_{k})} \in I_{1}\cup I_{2} $. 
\end{proof}\\
 accordance with the claim, $ y^{(2n)} $ will go to $   I_{1}\cup I_{2}  $ after several iterations . Thus,  sequence $ (0,y^{(2n)},z^{(2n)}) $ converges to $ (0,y^{\circ},1-y^{\circ}) $ whenever $ x^{(0)}=0 $. 
 
Let $ y^{(0)}=0 $ and  consider the following function:\\
\begin{equation}\label{tryew564604yjftju}
x^{(2)}= \vartheta(x^{(0)})=4x^{(0)}-10(x^{(0)})^2+8(x^{(0)})^3-2(x^{(0)})^4,
\end{equation}  where $ x^{(0)} \in  (0,1)$.  $ Fix(\vartheta) \cap (0,1)=\lbrace x^{\circ}\rbrace$ can be easily  shown. Through simple calculations, we determine   $\vartheta\left([0,1-\frac{1}{2}\sqrt{2}]  \right)\subseteq  [1-\frac{1}{2}\sqrt{2},1) $ and conclude that $ [1-\frac{1}{2}\sqrt{2},1) $ is sufficient to study the dynamics of $\vartheta  $  on $ (0,1) $.\\ 

To study the behavior of $ \vartheta$,  invariant interval $ [1-\frac{1}{2}\sqrt{2},1) $ is divided into three intervals as follows: $ I_{1}=[1-\frac{1}{2}\sqrt{2},x^{\circ}] $,  $I_{2}=[ x^{\circ},\frac{1}{2}]  $, and $I_{3}= [ \frac{1}{2},1 )   $. Thus, we have  two separate cases:
 \begin{enumerate} \item[(a)] Let $ x^{(0)}\in I_{1}  $, then $\vartheta(x^{(0)})\in I_{2}  $ and, $ \vartheta^{(2)}(x^{(0)})\in I_{1} $.  $\vartheta^{(2n+2)}(x^{(0)})\leq\vartheta^{(2n)}(x^{(0)})  $ whenever $ x^{(0)}\in I_{1}  $ can be easily checked. Therefore,  $\vartheta^{(2n)}  $ is a decreasing  bounded sequence that converges to a fixed point of $\vartheta^{(2)} $.  $ x^{\circ}  $ is  a fixed point of $\vartheta^{(2)} $ and  the only possible point of the convergence trajectory. Hence, $\vartheta^{(2)} $ converges to $ x^{\circ} $. 
\item[(b)] Similarly, let $ x^{(0)}\in I_{2}  $, then $\vartheta(x^{(0)})\in I_{1}  $ and $ \vartheta^{(2)}(x^{(0)})\in I_{2} $.  $\vartheta^{(2n+2)}(x^{(0)})\geq\vartheta^{(2n)}(x^{(0)})  $ whenever $ x^{(0)}\in I_{2}  $ can be easily checked. Therefore,  $\vartheta^{(2n)}  $ is an increasing  bounded sequence that converges to a fixed point of $\vartheta^{(2)} $.  $ x^{\circ}  $ is a fixed point of $\vartheta^{(2)}$ and the only possible point of the  convergence trajectory. Hence, $\vartheta^{(2n)} $ converges to $ x^{\circ} $. \end{enumerate}To explore the behavior of $ \vartheta $, when $ x^{(0)}\in I_{3} $,   the following claim is required:\\
 \begin{claim}
Let $ x^{(0)}\in I_{3} $ . Then,  $ n_{k}\in\mathbb N $, such that $ \vartheta^{(n_{k})} \in I_{1}\cup I_{2} $.
 \end{claim}
\begin{proof}
Let  $ x^{(0)}\in I_{3} $. Suppose that  interval $ I_{3} $ is invariant, which indicates that $ x^{(n)}\in I_{3} $ for any $ n\in\mathbb N $. Evidently,  $\vartheta ^{(n+1)}(x^{(0)})\leq  \vartheta ^{(n)}(x^{(0)}) $, which results in  sequence $\vartheta ^{(n)} $  being a decreasing  bounded and converging to a fixed point of $\vartheta  $. However, $ Fix(\vartheta )\cap I_{3} = \o $, which is contradiction. Hence,  $ n_{k}\in\mathbb N $,  such that $\vartheta^{( n_{k})} \in I_{1}\cup I_{2} $.
\end{proof}

accordance with the claim, $ x^{(n)} $ will go to $   I_{1}\cup I_{2}  $  after several iterations. Thus,  sequence $ (x^{(2n)},0,z^{(2n)}) $ converges to $ (x^{\circ},0,1-x^{\circ}) $ whenever $ y^{(0)}=0 $. In another way, if $ x^{(0)}=0 $, then \\\begin{equation}\label{fssfwrydeqqwddsa}
V^{(n)}(W_{2})=\left\{
\begin{array}{l}
\begin{split}
 (0,y^{\circ},1-y^{\circ})  \end{split}\ \ , if \ \ n=2k  \ \\
 \\
\begin{split}
\left(x^{\circ},0,1-x^{\circ}\right) 
\end{split}\ \ , if \ \ n=2k+1  \\ 
\end{array} \right. \\
\end{equation}\\ and if $ y^{(0)}=0$, then \\\begin{equation}\label{fssfwrydeqqwddsa}
V^{(n)}(W_{2})=\left\{
\begin{array}{l}
\begin{split}
 (0,y^{\circ},1-y^{\circ})  \end{split}\ \ , if \ \ n=2k+1  \ \\
 \\
\begin{split}
\left(x^{\circ},0,1-x^{\circ}\right) 
\end{split}\ \ , if \ \ n=2k  \\ 
\end{array} \right. \\
\end{equation}\\
From the preceding, we observe that if $ x^{(0)}=0 $ and $ n $ is an even, then the behavior of $W^{(2n)}_2  $ occurs in $ (0,y^{\circ},1-y^{\circ}) $,  which is equal to the behavior of $W_2  $ when $ y^{(0)}=0 $ and $ n $ is an odd iteration.  If $ y^{(0)}=0 $ and $ n $ is an even iteration, then the behavior of $W_2  $ occurs in $ (x^{\circ},0,1-x^{\circ}) $, which is equal to the behavior of $W_2  $ when $ x^{(0)}=0 $ and $ n $ is an odd iteration. Therefore, the limiting point of $W_2  $ consists of  $ (x^{\circ},0,1-x^{\circ}) $ and $ (0,y^{\circ},1-y^{\circ})$ whenever $ x^{(0)}\notin int\left(S^{2} \right)  $. If $ x^{(0)}=1 $, then the behavior of $ W_2  $ reaches  fixed point $ e_{3} $ after three iterations;  if $ y^{(0)}=1 $, then the behavior of $ W_2  $ reaches  fixed point $ e_{3} $ after one iteration. Therefore, the limiting point  in this case includes  $ e_{3} $, which is the desired conclusion.
 
\end{proof}\\Subsequently,  the behavior of operator $V_{15}  $ selected from class $ G_{9} $ is explored:\begin{equation}\label{606rtuyur0}
V_{15}:=\left\{
\begin{array}{l}
 x'=\left( 1-\alpha\right)  (x^{(0)})^{2}+2y^{(0)}z^{(0)}    \\
y'=\alpha   (x^{(0)})^{2}+2x^{(0)}\left( 1-x^{(0)}\right) \\
z'=(z^{(0)})^{2}+(y^{(0)})^{2}
\end{array} \right.
\end{equation}
The operator $V_{15}$ can be redrafted as a convex combination $ V_{15}= \left( 1-\alpha\right) W_{1}+\alpha  W_{2}$,\\
   where \\
 \begin{equation}\label{60udtr0}
W_{1}:=\left\{
\begin{array}{l}
x'=(x^{(0)})^{2}+2y^{(0)}z^{(0)} \\
y'=2x^{(0)}\left( 1-x^{(0)}\right)\\
z'=(z^{(0)})^{2}+(y^{(0)})^{2}
\end{array} \right.
\end{equation}\\
and
 \begin{equation}\label{60yfu1}
W_{2}:=\left\{
\begin{array}{l}
x'=2y^{(0)}z^{(0)}  \\
y'=2x^{(0)}-(x^{(0)})^{2}\\
z'=(z^{(0)})^{2}+(y^{(0)})^{2}
\end{array} \right.
\end{equation}
\begin{corollary}Let   $  W_{1}:S^{2}\rightarrow S^{2} $  be a $\xi^{(as)}$-QSO  given by \eqref{60udtr0}, and $ x_{1}^{(0)}=(x^{(0)},y^{(0)},z^{(0)})\notin Fix(W_{1})$  be any   initial point in  simplex $ S^{2} $. Then, the following statements are true:\\
\begin{enumerate}
\item[(i)]$ Fix(W_{1})=\left\{e_{1},e_{3},(\dfrac{3-\sqrt{3}}{4},\dfrac{\sqrt{3}}{4},\frac{1}{4})\right\rbrace  $
\item[(ii)] $ \omega _{w_{1}}(x_{1}^{(0)} )=\left\lbrace (\dfrac{3-\sqrt{3}}{4},\dfrac{\sqrt{3}}{4},\frac{1}{4})\right\rbrace  $,\\\end{enumerate} For $W_{2}  $  
 let   $  W_{2}:S^{2}\rightarrow S^{2} $  be a $\xi^{(as)}$-QSO  given by \eqref{60yfu1}and  $ x_{1}^{(0)}=(x^{(0)},y^{(0)},z^{(0)})\notin Fix(W_{2})$  be any   initial point in simplex $ S^{2} $. Then, the following statements are true:
 \begin{enumerate}
\item[(i)]$ Fix(W_{2})=\left\{e_{3},(x^{\centerdot},y^{\centerdot},z^{\centerdot})\right\rbrace  $,\\
\\
where $x^{\centerdot}=\frac{-1}{6}\sqrt[3]{t}-\frac{8}{3\sqrt[3]{t}+\frac{5}{3}}$, $y^{\centerdot}=\frac{-1}{6}\frac{3\sqrt{17}\sqrt[3]{t}+2\sqrt[3]{t^{2}}-24\sqrt{17}-5\sqrt[3]{t}-88}{\sqrt[3]{t^{2}}}  $,  $z^{\centerdot}=\frac{-1}{6}\frac{2\sqrt[3]{t^{2}}-3\sqrt{17}\sqrt[3]{t}-11\sqrt[3]{t}+6\sqrt{17}-10}{\sqrt[3]{t^{2}}}  ,$
\\
 and $ t=\left( 98+18\sqrt{17}\right)  $.
\item[(ii)] \begin{equation}\label{fssddsa}
Per_{2}(W_{2})=\left\{
\begin{array}{l}
\begin{split}
 e_{3},(0,y^{\circ},1-y^{\circ}) \end{split}\ \ , if \ \ x^{(0)}=0  \ \\
 \\
\begin{split}
e_{3},(x^{\circ},0,1-x^{\circ})
\end{split}\ \ , if \ \ y^{(0)}=0  \\ 
\end{array} \right. \\
\end{equation}
 where, $y^{\circ}=\dfrac{1}{6}(1+3\sqrt{57})^{\frac{1}{3}} -\frac{4}{3(1+3\sqrt{57})^{\frac{1}{3}}}+\frac{2}{3}$, $x^{\circ}=\dfrac{-1}{6}(46+6\sqrt{57})^{\frac{1}{3}}-\frac{2}{3(46+6\sqrt{57})^{\frac{1}{3}}}+\frac{4}{3}  $.
\item[(3)] \begin{equation}\label{fssdddswssa}
\omega _{w_{1}}(x_{1}^{(0)} )=\left\{
\begin{array}{l}
\begin{split}
 \quad\quad (x^{\centerdot},y^{\centerdot},z^{\centerdot}) \end{split}\ \ \quad\quad\quad\quad \quad\quad, if \ \ x_{1}^{(0)}\in int\left( S^{2}\right)   \ \\
 \\
\begin{split}
(x^{\circ},0,1-x^{\circ}),(0,y^{\circ},1-y^{\circ})
\end{split}\ \ , if \ x_{1}^{(0)} \notin int\left( S^{2}\right)   \\
\\
 \begin{split}
\quad\quad \quad e_{3} \end{split}\ \ \quad \quad\quad\quad\quad\textbf{}\quad\quad \quad, if \ \ x^{(0)}, y^{(0)}=1   \ \\
 \\ 
\end{array} \right.
\end{equation}

\end{enumerate}
\end{corollary}
\section{Dynamics of classes $G_{13}$ and  $G_{14}$} 
In this section, we study the dynamics of $ V_{26,25}:S^{2}\rightarrow S^{2}  $  selected from $ G_{14} $ and  $ G_{13} $. To start,   $V_{26} $ is rewritten as follows:\\\begin{equation}\label{60q[]tiopre0}
V_{26}:=\left\{
\begin{array}{l}
x'= (y^{(0)})^2+(z^{(0)})^2+2x^{(0)}\left( 1-x^{(0)}\right)   \\
y'= \alpha (x^{(0)})^2+2y^{(0)}z^{(0)} \\
z'=\left( 1-\alpha\right)  (x^{(0)})^2\\
\end{array} \right.
\end{equation}
The operator $V_{26}$ can be redrafted  as a convex combination $ V_{26}= \alpha  W_{1}+\left( 1-\alpha\right) W_{2} $,\\
   where \\
 \begin{equation}\label{6yiderot00}
W_{1}:=\left\{
\begin{array}{l}
x'=(y^{(0)})^2+(z^{(0)})^2+2x^{(0)}\left( 1-x^{(0)}\right)\\
y'= (x^{(0)})^2+2y^{(0)}z^{(0)}\\
z'=0
\end{array} \right.
\end{equation}\\
and
 \begin{equation}\label{6kjljyetd01}
W_{2}:=\left\{
\begin{array}{l}
x'=(y^{(0)})^2+(z^{(0)})^2+2x^{(0)}\left( 1-x^{(0)}\right)  \\
y'=2y^{(0)}z^{(0)}\\
z'=(x^{(0)})^2
\end{array} \right.
\end{equation}
\begin{theorem}\label{60ljukf./;2} Let   $  W_{1}:S^{2}\rightarrow S^{2} $  be a $\xi^{(as)}$-QSO  given by \eqref{6yiderot00} and $x_{1}^{(0)}=(x^{(0)},y^{(0)},z^{(0)})\notin Fix(W_{1})\cup Per_{2}(W_{1})$  be any   initial point in  simplex $ S^{2} $. Then, the following   statements are true:\\
\begin{enumerate}
\item[(i)]$ Fix(W_{1})=\left\lbrace \left(\frac{\sqrt{5}}{2}-\frac{1}{2},\frac{3}{2}-\frac{ \sqrt{5}}{2},0\right) \right\rbrace  $.\\ \item[(ii)]$Per_{2}(W_{1})=\left\lbrace e_{1},e_{2} ,(\frac{\sqrt{5}}{2}-\frac{1}{2},\frac{(-1+\sqrt{5})^2}{4},0)\right\rbrace   $,
\item[(iii)] $ \omega _{W_{1}}(x_{1}^{(0)} )=\left\lbrace e_{1},e_{3}\right\rbrace  $.
\end{enumerate}
\end{theorem}
\begin{proof}Let $  W_{1}:S^{2}\rightarrow S^{2} $ be a $\xi^{(as)}$-QSO  given by \eqref{6yiderot00},  $ x_{1}^{(0)}\notin Fix( W_{1})\cup Per_{2}(W_{1})$  be any  an initial point in  simplex $ S^{2} $, and   $ \left\lbrace W_{1}^{(n)} \right\rbrace _{n=1}^{\infty}  $
 be a trajectory of $  W_{1} $  starting from  point $x_{1}^{(0)}  $.\\
 
(1)  The set of the fixed points of $ W_{1} $ are obtained by finding the solution for the   following system of equations:\begin{equation}\label{60fhxsfgjhs2}
\left\{
\begin{array}{l}
x=y^{2}+z^{2}+2x\left( 1-x\right) \\
y=x^{2}+2yz\\
z=0
\end{array} \right.
\end{equation} By substituting the second and   third equations  \eqref{60fhxsfgjhs2} to the first equation, then the first equation in  system\eqref{60fhxsfgjhs2} becomes  $x^4-2x^2+x   $, then $ x=  0$, $ x=1$, and $x=\frac{\sqrt{5}}{2}-\frac{1}{2}  $.   $\frac{\sqrt{5}}{2}-\frac{1}{2}$ is verified as the only solution that satisfies  system  \eqref{60fhxsfgjhs2}. Hence, the fixed point is only $ \left(\frac{\sqrt{5}}{2}-\frac{1}{2},\frac{3}{2}-\frac{ \sqrt{5}}{2},0\right) $.\\\\(2) Let  $ x_{1}^{(0)}=\left( 1,0,0\right)   $ be the intial point.  $ V^{(1)}\left( x^{0},y^{0},z^{0}\right)=\left( 0,1,0\right) $ and  $ V^{(2)}\left( x^{0},y^{0},z^{0}\right)=\left( 1,0,0\right) $, which indicates the presence of  $ 2-$periodic points. To find all the  points, the following  system of equations should be solved :\begin{equation}\label{123456}
\left\{
\begin{array}{l}
x=2x^2-x^4 \\
y=(1- (1-y)^{2})^{2} \\
z=0
\end{array} \right.
\end{equation}\\From the first equation  in system \eqref{123456},  $ x\in \left\lbrace  0,1,\frac{\sqrt{5}}{2}-\frac{1}{2} \right\rbrace  $, then  $ y\in \left\lbrace  1,0, \frac{(-1+\sqrt{5})^2}{4} \right\rbrace  $. Therefore, $Per_{2}(W_{1})=\left\lbrace e_{1}=(1,0,0),e_{2}=(0,1,0) ,(\frac{\sqrt{5}}{2}-\frac{1}{2},\frac{(-1+\sqrt{5})^2}{4},0)\right\rbrace   $.\\\\(3) Let $ x_{1}^{(0)} \notin Fix(W_{1})\cup Per_{2}(W_{1}) $.  $L_{3}   $ is  an invariant line under $W_{1} $ where $L_{3}=\lbrace x_{1}^{(0)} \in S^{2}: z^{(0)}=0 \rbrace $. Thus,  the behavior of $W_{1} $ is explored over this line.  Let $  x_{1}^{(0)} \in L_{3}$. Then,  $W_{1} $ becomes:\begin{equation}\label{60fhxreyrssfgjhs2}
\left\{
\begin{array}{l}
x'=(y^{(0)})^{2}+2x^{(0)}\left( 1-x^{(0)}\right) \\
y'=(x^{(0)})^{2}\\
z'=0
\end{array} \right.
\end{equation}\\
In this case, the first coordinate of $ W_{1} $ exhibits the form $ x'=   \varphi (x^{(0)})=(1-x^{(0)})^{2}+2x^{(0)}(1-x^{(0)}) $. Clearly, the function $ \varphi   $ is decreasing on $[0,1]$ and the function $ \varphi  ^{(2)} $ is  increasing on $[0,1]$. From the previous two steps,  $Fix( \varphi)\cap [0,1]=\left\lbrace\frac{\sqrt{5}}{2}-\frac{1}{2} \right\rbrace   $ and  $Fix( \varphi^{(2)})\cap [0,1]=\left\lbrace0,\frac{\sqrt{5}}{2}-\frac{1}{2},1 \right\rbrace $, which indicate that  intervals $[0,\frac{\sqrt{5}}{2}-\frac{1}{2}]  $ and $[\frac{\sqrt{5}}{2}-\frac{1}{2},1]  $ are invariant under the function  $  \varphi^{(2)}$. Evidently, $ \varphi^{(2)}(x^{(0)})\leq x^{(0)}  $ for any  $ x^{(0)}\in [0,\frac{\sqrt{5}}{2}-\frac{1}{2}]    $ and  $ \varphi^{(2)}(x^{(0)})\geq x^{(0)}  $  for any  $ x^{(0)}\in [ \frac{\sqrt{5}}{2}-\frac{1}{2},1]$. If $ x^{(0)}\in [0,\frac{\sqrt{5}}{2}-\frac{1}{2}]  $, then $ \omega_{ \varphi^{(2)}}=\left\lbrace 0\right\rbrace  $; if $x^{(0)}\in [ \frac{\sqrt{5}}{2}-\frac{1}{2},1] $, then $ \omega_{ \varphi^{(2)}}=\left\lbrace 1\right\rbrace  $. In another way,\begin{equation}\label{fssfwqqwddsa}
V^{(n)}(W_{1})=\left\{
\begin{array}{l}
\begin{split}
\left(  \varphi^{(2k)}(x^{(0)}),1-  \varphi^{(2k)}(x^{(0)}),0\right)  \end{split}\ \ \quad \quad \quad, if \ \ n=2k  \ \\
\\
\begin{split}
\left(   \varphi^{(2k)}( \varphi(x^{(0)})),1-  \varphi^{(2k)}( \varphi(x^{(0)})),0\right) 
\end{split}\ \ , if \ \ n=2k+1  \\ 
\end{array} \right. \\
\end{equation}
Therefore, the limiting point is   $ \omega _{W_{1}}(x_{1}^{(0)} )=\left\lbrace e_{1},e_{2}\right\rbrace  $.
\end{proof}\\
\begin{theorem}\label{6tujed1}Let   $  W_{2}:S^{2}\rightarrow S^{2} $  be a $\xi^{(as)}$-QSO  given by \eqref{6kjljyetd01} and  $ x_{1}^{(0)}=(x^{(0)},y^{(0)},z^{(0)})\notin Fix(W_{2})\cup Per_{2}(W_{2})$  be any   initial point in  simplex  $ S^{2} $. Then, the following statements are true:
\begin{enumerate}
\item[(i)]$ Fix(W{2})=\O  $. Moreover, $Per_{2}(W_{2})=\left\lbrace e_{1},e_{3},(\frac{\sqrt{5}-1}{2},0,\frac{1}{16}(\sqrt{5}-3)^{2}(\sqrt{5}+1)^2) \right\rbrace  $.
\item[(ii)] $ \omega_{W_{2}}(x_{1}^{(0)})=\left\lbrace e_{1}, e_{3} \right\rbrace  $.
\end{enumerate}
\end{theorem}
\begin{proof} Let $  W_{2}:S^{2}\rightarrow S^{2} $ be a $\xi^{(as)}$-QSO  given by \eqref{6kjljyetd01},  $ x_{1}^{(0)}\notin Fix( W_{2})\cup Per_{2}(W_{2})$  be any   initial point in  simplex $ S^{2} $, and   $ \left\lbrace W_{2}^{(n)} \right\rbrace _{n=1}^{\infty}  $
 be a trajectory of $ W_{2}$  starting from  point $x_{1}^{(0)}  $.\\
 
(1)  The set of  fixed points of $ W_{2} $ are obtained by finding the solution for the   following system of equations:\begin{equation}\label{60f2}
\left\{
\begin{array}{l}
x=y^{2}+z^{2}+2x\left( 1-x\right) \\
y=2yz\\
z=x^{2}
\end{array} \right.
\end{equation}\\
The system provided by \eqref{60f2} has no solution on $ [0,1] $. Therefore, the set of fixed points is $\O  $. The second coordinate of $W_{2}  $ increases if $z^{(n)}\geq \frac{1}{2}  $ and decreases if $z^{(n)}\leq \frac{1}{2}  $. In both cases,  $W_{2}  $    has no any order periodic points in  set $W_{2}\setminus L_{2}  $ because the second coordinate of $W_{2}  $ increases or decreases along the iteration of $W_{2}\setminus L_{2}  $. Therefore,
finding $ 2- $periodic points of $W_{2}  $ over $L_{2}  $ is sufficient. To find  $ 2-$periodic points of $W_{2}  $, the following system of equations should be solved:\begin{equation}\label{asdf}
\left\{
\begin{array}{l}
x=x^{4}+2x^{2}(1-x^{2}) \\
y=0\\
z=( 1-(1-z)^2)^2
\end{array} \right.
\end{equation}\\ The solution for the first equation in  system \eqref{asdf} is easy to find. Therefore, the periodic points of $W_{2}  $ are $e_{1}=(1,0,0),e_{3}=(0,0,1)$, and $(\frac{\sqrt{5}-1}{2},0,\frac{1}{16}(\sqrt{5}-3)^{2}(\sqrt{5}+1)^2)  $.\\

(2) Let $ x_{1}^{(0)} \notin Fix(W_{2})\cup Per_{2}(W_{2}) $ and $ y^{(0)}=0$.   The first coordinate of $W_{2}  $ can be rewritten as $x'= (1-x^{(0)})^{2}+2x^{(0)}(1-x^{(0)}) $ because the second coordanate is invariant over $ L_{2} $.  The first coordinate is equal to the first coordinate of  $  W_{1} $, which has been proven in the previous thereom. Hence, we derive \begin{equation}\label{fssfwqqwddsa}
V^{(n)}(W_{2})=\left\{
\begin{array}{l}
\begin{split}
\left(   \varphi^{(2k)}(x^{(0)}),0,1-  \varphi^{(2k)}(x^{(0)})\right)  \end{split}\ \ \quad \quad \quad, if \ \ n=2k  \ \\
\\
\begin{split}
\left(  \varphi^{(2k)}( \varphi(x^{(0)})),0,1-  \varphi^{(2k)}( \varphi(x^{(0)}))\right) 
\end{split}\ \ , if \ \ n=2k+1  \\ 
\end{array} \right. \\
\end{equation}
Therefore, we determine that  $ \omega _{W_{2}}(x^{(0)} )=\left\lbrace e_{1},e_{3}\right\rbrace  $. Let $ y^{(0)}\notin L_{2}$ and  $ x^{(n)} < \frac{1}{2}$, which indicate that $z^{(n)}< \frac{1}{2}$ and  yields  $y^{(n+1)}< y^{(n)}  $. If  $x^{(n)}< \frac{1}{2}$, then the third coordinate $ z^{(n)} $ is also less than  $  \frac{1}{2} $, which indicates that  $y^{(n+1)}< y^{(n)}  $. In the two previous cases, we conclude that $\dfrac{y^{(n+1)}}{y^{(n)}}\leq 1 $, thereby making $y^{(n+1)}  $ is a decreasing bounded sequence that converges to zero, which indicates that  studying  the dynamics of $ W_{2} $ over $ L_{2} $ was enough. Therefore,  $ \omega _{W_{2}}(x_{1}^{(0)} )=\left\lbrace e_{1},e_{3}\right\rbrace  $ for any initial point $ x_{1}^{(0)} $ in $ S^{2} $.
\end{proof}
\\

Subsequently, we  explore the behavior of $V_{25}  $, which is selected from class $ G_{9}$.         \begin{equation} \label{60q[]tiopre0}
V_{25}:=\left\{
\begin{array}{l}
 x'= (y^{(0)})^{2}+(y^{(0)})^{2}+2x^{(0)}\left( 1-x^{(0)}\right)   \\
y'= \alpha (x^{(0)})^{2}\\
z'=\left( 1-\alpha\right)  (x^{(0)})^{2}+2y^{(0)}z^{(0)} \\
\end{array} \right.\
\end{equation}
  We rewrite $V_{25}$ as a convex combination $ V_{25}= \alpha  W_{1}+\left( 1-\alpha\right) W_{2} $,\\
   where \\
 \begin{equation}\label{6yiwdqat00}
W_{1}:=\left\{
\begin{array}{l}
x'=(y^{(0)})^{2}+(y^{(0)})^{2}+2x^{(0)}\left( 1-x^{(0)}\right)\\
y'=(x^{(0)})^{2}\\
z'=2y^{(0)}z^{(0)}
\end{array} \right.
\end{equation}\\
and
 \begin{equation}\label{6fwqatd01}
W_{2}:=\left\{
\begin{array}{l}
x'=(y^{(0)})^{2}+(y^{(0)})^{2}+2x^{(0)}\left( 1-x^{(0)}\right)  \\
y'=0\\
z'=(x^{(0)})^{2}+2y^{(0)}z^{(0)}
\end{array} \right.
\end{equation}
\begin{corollary}\label{60ljukf./;2}  Let $  W_{1}:S^{2}\rightarrow S^{2} $ given by \eqref{6yiwdqat00} be a $\xi^{(as)}$-QSO. Then,  the following   statements are true:
\begin{enumerate}
\item[(i)]$ Fix(W_{1})=\O  $. Moreover, $Per_{2}(W_{1})=\left\lbrace e_{1},e_{2},(\frac{\sqrt{5}-1}{2},\frac{1}{16}(\sqrt{5}-3)^{2}(\sqrt{5}+1)^2,0) \right\rbrace  $.
\item[(ii)] $ \omega_{W_{1}}(x_{1}^{(0)})=\left\lbrace e_{1}, e_{2} \right\rbrace  $
\end{enumerate} For $ W_{2} $,  
let $  W_{2}:S^{2}\rightarrow S^{2} $ given by \eqref{6fwqatd01} be a $\xi^{(as)}$-QSO. Then, the following   statements are true:
\begin{enumerate}
\item[(i)]$ Fix(W_{2})=\left\lbrace \left(\frac{\sqrt{5}}{2}-\frac{1}{2},0,\frac{3}{2}-\frac{ \sqrt{5}}{2}\right) \right\rbrace  $\\ \item[(ii)]$Per_{2}(W_{2})=\left\lbrace e_{1},e_{3} \right\rbrace   $
\item[(iii)] $ \omega _{w_{2}}(x_{1}^{(0)} )=\left\lbrace e_{1},e_{2}\right\rbrace  $
\end{enumerate}
\end{corollary}

%\section{Conclusion}
%In this paper, it was introduced new partitions of classes  $\xi^{(as)}$-QSO on 2D simplex and found classes of $\xi^{(as)}$-QSO. Moreover, we classified these classes into $ 18 $-non conjugate  classes. Furthermore, we examined the dynamic  four   classes of $\xi^{(as)}$-QSO 

%{\bf Acknowledgements}

\end{document}